\documentclass[a4paper,11pt,reqno]{article}

\usepackage[utf8]{inputenc}
\usepackage[T1]{fontenc}
\usepackage{lmodern}
\usepackage[english]{babel}
\usepackage{microtype}
\usepackage{csquotes}

\usepackage[backend=biber,hyperref=true,style=alphabetic,giveninits=true]{biblatex}
\usepackage{biblatex-mr}
\usepackage{tikz}
\usetikzlibrary{math}

\definecolor{midnightblue}{rgb}{0.1, 0.1, 0.44}
\usepackage[colorlinks,allcolors=midnightblue,pdfusetitle]{hyperref}
\AtBeginDocument{}
\AtBeginDocument{}
\AtBeginDocument{}
\hypersetup{
	pdfsubject={Particle method for nonlinear nonlocal scalar balance equations},
	pdfkeywords={scalar balance law, nonlinear, nonlocal, hyperbolic equation, entropy solution,
		numerical scheme, particle method, Lagrangian approximation}}

\usepackage[a4paper,hmargin=2.5cm,bottom=3cm,top=3cm,footskip=3\baselineskip,
	marginparwidth=2.4cm,marginparsep=0.05cm]{geometry}

\usepackage{fsmath}
\usepackage{bm}
\usepackage[normalem]{ulem}
\usepackage{footmisc}
\usepackage{enumitem}
\setlist[enumerate,1]{label=(\roman*)}

\newcommand{\loc}{\mathrm{loc}}
\DeclareMathOperator{\TV}{TV}

\newcommand{\dx}{\d x}
\newcommand{\dy}{\d y}

\newcommand{\dt}{\d t}
\newcommand{\ds}{\d s}

\newcommand{\g}{\mathfrak{g}}
\newcommand{\G}{\mathcal{G}}
\newcommand{\U}{\mathcal{U}}
\newcommand{\V}{\mathcal{V}}
\newcommand{\W}{\mathcal{W}}

\newcommand{\II}{I\!I}

\begin{document}

\title{Deterministic particle method for \\ nonlinear nonlocal scalar balance equations}
\author{
	Emanuela Radici\thanks{Università degli Studi dell'Aquila, \url{emanuela.radici@univaq.it}}
	\and
	Federico Stra\thanks{\url{stra.federico@gmail.com}}
}

\maketitle

\begin{abstract}
	We study a deterministic particle scheme to solve a scalar balance equation with nonlocal interaction
	and nonlinear mobility used to model congested dynamics.
	The main novelty with respect to ``Radici--Stra [SIAM J.Math.Anal 55.3 (2023)]'' is the presence of a
	source term; this causes the solutions to no longer be probability measures, thus requiring a suitable
	adaptation of the numerical scheme and of the estimates leading to compactness.
\end{abstract}

\tableofcontents

\section{Introduction}

This work is devoted to the study of a deterministic particle scheme approximating an entropy solution
of the scalar balance equation
\begin{equation}\label{eq:pde}
	\de_t\rho(t,x) + \de_x\bigl[
		\rho(t,x) \underbrace{v\bigl(\rho(t,x)\bigr)}_\text{congestion} \bigl(
		\underbrace{V(t,x)}_\text{advection}
		- \underbrace{(\de_x W*\rho)(t,x)}_\text{interaction}
		\bigr)\bigr]
	= \underbrace{f\bigl(t,x,\rho(t,x)\bigr)}_\text{source}
	.
\end{equation}
The PDE \eqref{eq:pde} models the evolution of a density $\rho$ under the effect of
an external advection field $V$, a nonlocal self interaction potential $W$, and a source term $f$.
The intensity of the free velocity field $U=V-\de_xW*\rho$ is modulated by a congestion penalization
$v(\rho)$ which is a non-increasing function of the pointwise density. The product $m(\rho)=\rho v(\rho)$
is called mobility and in our setting it can be nonlinear.

This PDE finds applications in various fields ranging from chemistry, to biology and social/economical
sciences, as several models used to describe the macroscopic evolution of many interacting agents
capturing behaviors such as queue and pattern formation fall into this category.

The numerical scheme is inspired by a sequence of works for similar conservation laws without the
source term.
In particular we mention \cite{DiFrancesco-Rosini,DiFrancesco-Stivaletta} for the case $W=0$ and
\cite{DiFrancesco-Fagioli-Radici,Fagioli-Tse,Radici-Stra-2023} for evolutions with self interaction.
Along this line of research there are also results for the conservation law with diffusion instead of the
source term, see for instance \cite{Fagioli-Radici,Daneri-Radici-Runa-2023}.

The evolution of a solution $\rho$ of the PDE \eqref{eq:pde} is approximated by piecewise constant
densities $\bar\rho^N$ constructed as follows. For $N\in\setN_+$, given $N+1$ sorted particles
$(x_0,\dots,x_N)$ and $N$ non-negative masses $(q_1,\dots,q_N)$, define the associated piecewise constant
density
\begin{subequations}\label{eq:rhobar}
	\begin{equation}\label{eq:rhobar-rho}
		\bar\rho^N(t,x) = \sum_{i=1}^N \rho_i(t) \bm1_{\bigl(x_{i-1}(t),x_i(t)\bigr)},
	\end{equation}
	where
	\begin{align}\label{eq:rho_i}
		\rho_i(t) & = \frac{q_i(t)}{x_i(t)-x_{i-1}(t)} \quad \forall i\in\{1,\dots,N\}, &
		\rho_0(t) & = \rho_{N+1}(t) = 0;
	\end{align}
	define also the total mass
	\begin{equation}\label{eq:q}
		q^N(t) = \norm{\bar\rho^N(t,\plchldr)}_{L^1(\setR)} = \sum_{i=1}^N q_i(t).
	\end{equation}
\end{subequations}
\begin{figure}
	\centering
	\begin{tikzpicture}
		\tikzmath{
			\ts = 0.1;
			\x0 = -5.5; \x1 = -4; \x2 = -3;
			\xim2 = -2; \xim1 = -0.75; \xii = 0.5; \xip1 = 1.5;
			\xNm2 = 2.5; \xNm1 = 3.5; \xN = 5;
			\r1 = 0.8; \r2 = 1.2;
			\rim1 = 1; \ri = 1.5; \rip1 = 2;
			\rNm1 = 1.5; \rN = 0.75;
			\xmin = \x0-2; \xmax = \xN + 2;
		}

		\draw[-stealth] (\xmin,0) -- (\xmax,0); 

		\draw 
		(\x0,-\ts) node[below] {$x_0$} -- +(0,\ts)
		(\x1,-\ts) node[below] {$x_1$} -- +(0,\ts)
		(\x2,-\ts) node[below] {$x_2$} -- +(0,\ts)
		(\xim2,-\ts) node[below] {$x_{i-2}$} -- +(0,\ts)
		(\xim1,-\ts) node[below] {$x_{i-1}$} -- +(0,\ts)
		(\xii,-\ts) node[below] {$x_i$} -- +(0,\ts)
		(\xip1,-\ts) node[below] {$x_{i+1}$} -- +(0,\ts)
		(\xNm2,-\ts) node[below] {$x_{N-2}$} -- +(0,\ts)
		(\xNm1,-\ts) node[below] {$x_{N-1}$} -- +(0,\ts)
		(\xN,-\ts) node[below] {$x_N$} -- +(0,\ts)
		;

		\draw[dashed] 
		(\x0,0) -- +(0,\r1)
		(\x1,0) -- +(0,\r2)
		(\x2,0) -- +(0,\r2)
		(\xim2,0) -- +(0,\rim1)
		(\xim1,0) -- +(0,\ri)
		(\xii,0) -- +(0,\rip1)
		(\xip1,0) -- +(0,\rip1)
		(\xNm2,0) -- +(0,\rNm1)
		(\xNm1,0) -- +(0,\rNm1)
		(\xN,0) -- +(0,\rN)
		;

		\draw[thick] 
		(\x0-1,0) node[above] {$\rho_0$}
		(\x0,\r1) -- (\x1,\r1) node[pos=0.5,above] {$\rho_1$}
		(\x1,\r2) -- (\x2,\r2) node[pos=0.5,above] {$\rho_2$}
		(\xim2,\rim1) -- (\xim1,\rim1) node[pos=0.5,above] {$\rho_{i-1}$}
		(\xim1,\ri) -- (\xii,\ri) node[pos=0.5,above] {$\rho_i$}
		(\xii,\rip1) -- (\xip1,\rip1) node[pos=0.5,above] {$\rho_{i+1}$}
		(\xNm2,\rNm1) -- (\xNm1,\rNm1) node[pos=0.5,above] {$\rho_{N-1}$}
		(\xNm1,\rN) -- (\xN,\rN) node[pos=0.5,above] {$\rho_N$}
		(\xN+1,0) node[above] {$\rho_{N+1}$}
		;

		\path 
		(\x0,0) -- (\x1,0) node[pos=0.5,above] {$q_1$}
		(\x1,0) -- (\x2,0) node[pos=0.5,above] {$q_2$}
		(\xim2,0) -- (\xim1,0) node[pos=0.5,above] {$q_{i-1}$}
		(\xim1,0) -- (\xii,0) node[pos=0.5,above] {$q_i$}
		(\xii,0) -- (\xip1,0) node[pos=0.5,above] {$q_{i+1}$}
		(\xNm2,0) -- (\xNm1,0) node[pos=0.5,above] {$q_{N-1}$}
		(\xNm1,0) -- (\xN,0) node[pos=0.5,above] {$q_N$}
		;

		\path 
		(\x2,0) -- (\xim2,0) node[pos=0.5,above] {\dots}
		(\xip1,0) -- (\xNm2,0) node[pos=0.5,above] {\dots}
		;
	\end{tikzpicture}
	\caption{Piecewise constant density $\bar\rho^N$ associated to particles $(x_0,\dots,x_N)$ and masses
		$(q_1,\dots,q_N)$.}
	\label{fig:rhobar}
\end{figure}
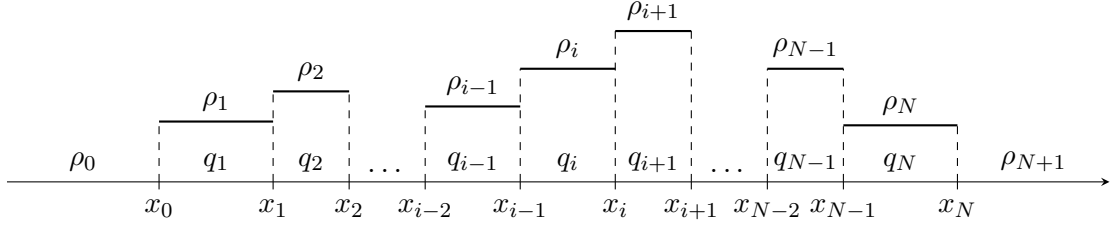
In \autoref{fig:rhobar} you can see an illustration of a piecewise constant density with its defining
parameters.
The initial datum is chosen using \autoref{lem:initial-datum}.
In order to guarantee that the evolution of $\bar\rho^N(t,\plchldr)$ approximates the PDE \eqref{eq:pde},
the particles $x_i(t)$ and the masses $q_i(t)$ are subjected to the system of ODEs
\begin{subequations}
	\label{eq:ode}
	\begin{align}
		\label{eq:ode-x}
		x'_i(t) & = v_i(t) U_i(t)
		        & \forall i\in\{0,\dots,N\}, \\
		\label{eq:ode-q}
		q'_i(t) &
		= \int_{x_{i-1}(t)}^{x_i(t)} f\bigl(t,x,\rho_i(t)\bigr)\dx
		        & \forall i\in\{1,\dots,N\},
	\end{align}
	where for $i\in\{0,\dots,N\}$ we have defined the quantities
	\begin{align}
		\label{eq:ode-U}
		U_i(t) & = V\bigl(t,x_i(t)\bigr) - (\de_x W*\bar\rho^N)\bigl(t,x_i(t)\bigr),  \\
		\label{eq:ode-v}
		v_i(t) & = \begin{cases}
			           v\bigl(\rho_i(t)\bigr),     & \text{if}\ U_i(t)<0,    \\
			           v\bigl(\rho_{i+1}(t)\bigr), & \text{if}\ U_i(t)\geq0.
		           \end{cases}
	\end{align}
\end{subequations}

Our main goal is to show that the density $\rho$ obtained as a limit of $\bar\rho^N$ is an \emph{entropy
	solution} of the PDE \eqref{eq:pde} in the sense of \autoref{def:entropy-solution} (see \cite{Kruzkov}).
The reason is that it is well known that weak solutions are not unique in general for a PDE of the form
\eqref{eq:pde}. On the other hand, uniqueness of entropy solutions is known for a variety of classes of
PDEs which include the one we study in the case $f=0$; see for instance
\cite{Colombo-Mercier-Rosini,Karlsen-Risebro-2003,Marconi-Radici-Stra-2022}. All the uniqueness and
stability results currently available are insufficient to deduce the uniqueness of entropy solutions in
our more general setting. We conjecture that it should hold and can be obtained with a suitable
modification of \cite{Marconi-Radici-Stra-2022}, which we plan on doing in the future.

\newcommand{\tmpcite}{\cite{Kruzkov}}
\begin{definition}[Entropy solution, \tmpcite]\label{def:entropy-solution}
	Let $v,V,W,f$ satisfy the \autoref{ass}.

	We say that a non-negative function
	$\rho\in C\bigl([0,T);L^1_\loc(\setR)\bigr)\cap L^\infty_\loc\bigl([0,T);BV(\setR)\bigr)$
	is an entropy solution of \eqref{eq:pde} if, letting $U=V-\de_xW*\rho$, the
	following entropy inequality
	\begin{equation*}
		\begin{split}
			\int_0^T\!\!\int_\setR &\Bigl\{
			\abs{\rho-c}\de_t\phi + \sign(\rho-c)\bigl[\bigl(m(\rho)-m(c)\bigr)U(t,x)\de_x\phi
				-m(c)\de_xU(t,x)\phi + f(t,x,\rho)\phi
				\bigr]
			\Bigr\} \dx\dt
			\geq 0
		\end{split}
	\end{equation*}
	holds for every constant $c\geq0$ and non-negative test function
	$\phi\in C^\infty_c\bigl([0,T)\times\setR;[0,\infty)\bigr)$.
\end{definition}

With this notion of solution we can now state our main result. The assumptions that we make on the
congestion $v$, the field $V$, the potential $W$, and the source term $f$ are listed separately right
afterwards.

\begin{theorem}\label{thm:main}
	Let $v,V,W,f$ satisfy the \autoref{ass} and let $\rho_0\in\Meas_+(\setR)\cap L^\infty(\setR)\cap
		BV(\setR)$ be a compactly supported non-negative measure.
	For every $N\in\setN_+$, let $\bar\rho^N$ be the piecewise constant density associated via
	\eqref{eq:rhobar} to particles $(x_0,\dots,x_N)$ and masses $(q_1,\dots,q_N)$ solving the system of
	ODEs \eqref{eq:ode} starting from an initial condition given by \autoref{lem:initial-datum}.

	Then, up to subsequence, $(\bar\rho^N)_{N\in\setN_+}$ converges in
	$L^1_\loc\bigl([0,\infty)\times\setR\bigr)$ to a non-negative density
	\[
		\rho \in L^\infty_\loc\bigl([0,\infty)\times\setR\bigr)
		\cap BV_\loc\bigl([0,\infty)\times\setR\bigr)
		\cap C\bigl([0,\infty);L^1(\setR)\bigr)
	\]
	which is an entropy solution of \eqref{eq:pde} in the sense of \autoref{def:entropy-solution} with
	initial datum $\rho_0$. Moreover $\rho$ enjoys the properties
	\begin{align*}
		\norm{\rho(t,\plchldr)}_{L^1(\setR)} & \leq \norm{\rho_0}_1 Q(t), &
		\rho(t,\plchldr)                     & \leq R(t),                   \\
		\supp \rho(t,\plchldr)               & \subseteq [-S(t),S(t)],    &
		\TV\bigl(\rho(t,\plchldr)\bigr)      & \leq B(t),
	\end{align*}
	for every $t\in[0,\infty)$ for some increasing functions $Q,R,S,B:[0,\infty)\to[0,\infty)$.
\end{theorem}

\begin{assumptions}\label{ass}
	There exist three non-decreasing functions $F,G,\lambda : [0,\infty)\to[0,\infty)$
	such that:
	\begin{enumerate}[label=\textup{(A\arabic*)}]
		\item \label{as:v}
		      $v\in W^{1,\infty}_\loc\bigl([0,\infty);[0,\infty)\bigr)$ is a non-increasing function satisfying
		      \[
			      \norm{v'}_{L^\infty([0,r])} \leq G(r) \qquad\forall r>0;
		      \]
		\item \label{as:V}
		      $V(t,\plchldr)\in W^{2,1}_\loc(\setR)$, with the estimates for every $t>0$ and $x>0$
		      \begin{align*}
			      \norm{V(t,\plchldr)}_{W^{1,\infty}([-x,x])} & \leq F(t)G(x), &
			      \norm{\de_x^2 V(t,\plchldr)}_{L^1([-x,x])}  & \leq F(t)G(x);
		      \end{align*}
		\item \label{as:W}
		      the function $W(t,\plchldr)$ has the following regularity in space:
		      \begin{itemize}
			      \item $W(t,\plchldr)\in W^{1,\infty}_\loc(\setR)$,
			      \item $\de_x W(t,\plchldr)\in BV_\loc(\setR)$,
			      \item $D_x\de_x W(t,\plchldr) = \de_x^2 W(t,\plchldr)\leb^1 + w(t)\delta_0$,
			      \item $\de_x^2 W(t,\plchldr)\in W^{1,1}_\loc\bigl((-\infty,0]\bigr)\cap W^{1,1}_\loc\bigl([0,\infty)\bigr)$ with weak derivative denoted by $\de_x^3 W$,
		      \end{itemize}
		      with the estimates for every $t>0$ and $x>0$
		      \begin{align*}
			      \norm{\de_x W(t,\plchldr)}_{L^\infty([-x,x])}   & \leq F(t)G(x) , &
			      \abs{w(t)}                                      & \leq F(t) ,       \\
			      \norm{\de_x^2 W(t,\plchldr)}_{L^\infty([-x,x])} & \leq F(t)G(x) , &
			      \norm{\de_x^3 W(t,\plchldr)}_{L^1([-x,x])}      & \leq F(t)G(x) ;
		      \end{align*}
		\item \label{as:mild-growth}
		      for every $t>0$ and $x\in\setR$ we have
		      \begin{align*}
			      \sign(x)V(t,x)       & \leq F(t)\lambda(x), &
			      -\sign(x)\de_xW(t,x) & \leq F(t)\lambda(x)
		      \end{align*}
		      and the function $1/\lambda$ is not integrable at infinity, i.e.\ for every $x_0>0$ we have
		      \[
			      \int_{x_0}^\infty \frac{\dx}{\lambda(x)} = \infty;
		      \]
		\item \label{as:no-collapse}
		      at least one of the two following options holds:
		      \begin{enumerate}[label=\textup{(A\arabic{enumi}$_v$\hspace{-.2ex})}]
			      \item \label{as:v-decays}
			            $v$ decays sufficiently fast so that there exists a non-decreasing function
			            $g:[0,\infty)\to[0,\infty)$ such that $r+r^2v(r) \leq g(r)$ and $1/g$ is not integrable
			            at infinity, i.e.\ for every $r_0>0$ we have
			            \[
				            \int_{r_0}^\infty \frac{\d r}{g(r)} = \infty;
			            \]
		      \end{enumerate}
		      \begin{enumerate}[label=\textup{(A\arabic{enumi}$_W$\!)}]
			      \item \label{as:W-repulsive}
			            $W$ is expansive at small scales, i.e.\ $w(t)\leq0$ for every $t>0$;
		      \end{enumerate}
		\item \label{as:f}
		      there is a non-negative constant $c_f\in[0,\infty)$ such that
		      for every $t>0$ and $x\in\setR$ we have
		      \[
			      \abs{f(t,x,\rho)} \leq c_f F(t)\rho;
		      \]
		      moreover for every $t>0$ and $r>0$ there is a locally finite non-negative Borel measure
		      $\eta_{t,r}\in\Meas_+(\setR)$ such that
		      \begin{align*}
			      \abs{D_xf(t,\plchldr,\rho)}                       & \leq \eta_{t,r} \quad\forall\rho\in[0,r], &
			      \norm{\de_\rho f(t,x,\plchldr)}_{L^\infty([0,r])} & \leq F(t) G(r).
		      \end{align*}
	\end{enumerate}
\end{assumptions}

\section{Theoretical analysis}

This section is structured as follow:
\begin{itemize}

	\item in \autoref{sub:a-priori} we prove uniform a priori bounds for the total mass, support, density
	      and $BV$ norm of the numerical approximations $\bar\rho^N$, which in particular are required to
	      obtain their global existence in time;

	\item in \autoref{sub:compactness} we show that the uniform bounds coupled with an equicontinuity in
	      time with respect to a suitable weak topology are sufficient to obtain the compactness
	      in $L^1_\loc\bigl([0,\infty)\times\setR\bigr)$ of the family $(\bar\rho^N)_{N\in\setN_+}$,
	      thus deducing the existence of a non-negative limit density $\rho$;

	\item in \autoref{sub:entropy} we show that the numerical approximations $\bar\rho^N$ satisfy an
	      approximate entropy inequality with an error term that vanishes as $N\to\infty$
	      and we pass to the limit this information to deduce that $\rho$ is an entropy solution
	      of \eqref{eq:pde}.

\end{itemize}

We start with a preliminary lemma explaining how we select an initial condition $\bar\rho^N(0,\plchldr)$
for the ODE \eqref{eq:ode} from the initial datum $\rho_0$ of the PDE \eqref{eq:pde}.

\begin{lemma}[Initial datum]\label{lem:initial-datum}
	Given a non-negative density $\rho \in L^\infty(\setR;[0,\infty)) \cap BV(\setR)$ with compact support,
	for every $N\in\setN_+$ there exist $N+1$ sorted particles $(x_0,\dots,x_N)$ and $N$ non-negative masses
	$(q_1,\dots,q_N)$ such that the associated piecewise constant density $\bar\rho^N$ satisfies
	\begin{gather*}
		\supp(\bar\rho^N) = \operatorname{conv}\bigl(\supp(\rho)\bigr), \qquad\qquad
		\bar\rho^N \leq \norm{\rho}_{L^\infty(\setR)}, \qquad\qquad
		\TV(\bar\rho^N) \leq \TV(\rho), \\
		\lim_{N\to\infty} \norm{\bar\rho^N-\rho}_{L^1(\setR)} = 0, \hspace{2cm}
		\lim_{N\to\infty} \max_{i\in\{1,\dots,N\}} q_i = 0.
	\end{gather*}
\end{lemma}

\begin{proof}
	Let $\operatorname{conv}\bigl(\supp(\rho)\bigr) = [x_0,x_N]$. For every $i\in\{1,\dots,N-1\}$ define
	\[
		x_i = \min\set*{x\in[x_0,x_N]}{\rho([x_0,x]) = \frac iN \norm{\rho}_1},
	\]
	and for every $i\in\{1,\dots,N\}$ let $q_i = \norm{\rho}_1/N$.

	The fact that $\supp(\bar\rho^N) = \operatorname{conv}\bigl(\supp(\rho)\bigr)$ and $q_i\to0$ are
	immediate. The other properties can be obtained following the same proof of
	\cite[Lemma~1.2]{Radici-Stra-2023} or \cite[Lemma~3.9]{Marconi-Radici-Stra-2022}.
\end{proof}

Notice that as an immediate consequence of \eqref{eq:rho_i} we have that the derivative of the density
consists of two contributions arising from the advection and the source term respectively:
\begin{equation}\label{eq:ode-rho}
	\begin{split}
		\rho_i'(t)
		&= -q_i(t)\frac{x_i'(t)-x_{i-1}'(t)}{\bigl(x_i(t)-x_{i-1}(t)\bigr)^2} + \frac{q_i'(t)}{x_i(t)-x_{i-1}(t)} \\
		&= -\rho_i(t)\frac{x_i'(t)-x_{i-1}'(t)}{x_i(t)-x_{i-1}(t)} + \frac{q_i'(t)}{x_i(t)-x_{i-1}(t)} \\
		&= -\rho_i(t)\frac{v_i(t)U_i(t)-v_{i-1}(t)U_{i-1}(t)}{x_i(t)-x_{i-1}(t)}
		+ \dashint_{x_{i-1}(t)}^{x_i(t)} f\bigl(t,x,\rho_i(t)\bigr)\dx \\
		&= \rho_i'^{\mathrm{adv}}(t) + \rho_i'^{\mathrm{src}}(t).
	\end{split}
\end{equation}

\subsection{Properties of the numerical approximation\texorpdfstring{ $\bar\rho^N$}{}}\label{sub:a-priori}

In this subsection we study the properties of a single approximate solution with $N\in\setN$ fixed,
so for simplicity of notation we drop the explicit dependence on $N$ and write $\bar\rho=\bar\rho^N$ instead.

From the \autoref{ass} it is simple to verify that the right hand side of \eqref{eq:ode} is locally Lipschitz
with respect to the particles $(x_0,\dots,x_N)$ and the masses $(q_1,\dots,q_N)$;
therefore we have local existence and uniqueness of solutions of \eqref{eq:ode} for any initial datum
by Carathéodory theorem. We refer to \cite[Lemma~2.1]{Radici-Stra-2023} for the details of the proof
(the source term is not problematic because the assumption \ref{as:f} is enough to deduce the desired
regularity w.r.t.\ $x_i$ and $q_i$).

All the quantities involved in the estimates of the two following lemmas are computed at a fixed time,
therefore for notational convenience we omit the explicit time dependence.

The following lemma encapsulates the crucial consequences of assumption \ref{as:v} (the monotonicity of $v$) and of the specific choice of the downstream density in the computation of the congestion \eqref{eq:ode-v}.
We would like to point out that the choice of the downstream density in \eqref{eq:ode-v} is the crucial
ingredient that makes the resulting scheme be approximately entropic; without that there would be no hope
of converging to an entropy solution (also, some of the a priori estimates would fall apart).

\begin{lemma}\label{lem:good-v}
	Let $v$ satisfy \ref{as:v} and let $x_i'$ and $v_i$ be defined by \eqref{eq:ode-x} and
	\eqref{eq:ode-v} respectively.
	Define $\sigma_i=\sign(\rho_{i+1}-\rho_i)$ and $\Delta\sigma_i=\sigma_i-\sigma_{i-1}$. Then
	\begin{align}
		\refstepcounter{equation}
		\tag{{\theequation}a}
		\label{eq:good-v-max}
		\rho_i\geq\rho_{i\pm1} \implies & x_i'-x_{i-1}' \geq v(\rho_i)(U_i-U_{i-1}),
		                                &
		                                & \forall i\in\{1,\dots,N\},                             \\
		\tag{{\theequation}b}
		\label{eq:good-v-min}
		\rho_i\leq\rho_{i\pm1} \implies & x_i'-x_{i-1}' \leq v(\rho_i)(U_i-U_{i-1}),
		                                &
		                                & \forall i\in\{1,\dots,N\},                             \\
		\label{eq:good-v-step}
		\Delta\sigma_i(x_i'-x_{i-1}')   & \leq \Delta\sigma_i v(\rho_i) (U_i-U_{i-1}),
		                                &
		                                & \forall i\in\{1,\dots,N\},                             \\
		\label{eq:good-v-c}
		\bigl(\sign(\rho_{i+1}-c)-      & \sign(\rho_i-c)\bigr) \bigl(v_i-v(c)\bigr) U_i \leq 0,
		                                &
		                                & \forall i\in\{0,\dots,N\}.
	\end{align}
\end{lemma}

\begin{proof}
	If $\rho_i\geq\rho_{i\pm1}$, then
	\[
		\left.
		\begin{aligned}
			U_i & \geq0 &  & \implies &  & v_i-v(\rho_i) = v(\rho_{i+1})-v(\rho_i) \geq 0 \\
			U_i & <0    &  & \implies &  & v_i-v(\rho_i) = v(\rho_i)-v(\rho_i) = 0
		\end{aligned}
		\right\} \implies \bigl(v_i-v(\rho_i)\bigr)U_i \geq 0
	\]
	and similarly
	\[
		\left.
		\begin{aligned}
			U_{i-1} & \geq0 &  & \implies &  & v_{i-1}-v(\rho_i) = v(\rho_i)-v(\rho_i) = 0        \\
			U_{i-1} & <0    &  & \implies &  & v_{i-1}-v(\rho_i) = v(\rho_{i-1})-v(\rho_i) \geq 0
		\end{aligned}
		\right\} \implies \bigl(v_{i-1}-v(\rho_i)\bigr)U_{i-1} \leq 0,
	\]
	therefore
	\[
		\begin{split}
			x_i'-x_{i-1}'
			= v(\rho_i)(U_i-U_{i-1}) + \bigl(v_i-v(\rho_i)\bigr)U_i - \bigl(v_{i-1}-v(\rho_i)\bigr)U_{i-1}
			\geq v(\rho_i)(U_i-U_{i-1}),
		\end{split}
	\]
	which proves \eqref{eq:good-v-max}; the proof of \eqref{eq:good-v-min} is analogous.

	Let us now move on to \eqref{eq:good-v-step}. If $\rho_i\geq\rho_{i\pm1}$, then $\Delta\sigma_i\leq0$
	and the inequality follows from \eqref{eq:good-v-max}; similarly, if $\rho_i\leq\rho_{i\pm1}$, then
	$\Delta\sigma_i\leq0$ and the inequality follows from \eqref{eq:good-v-min}. If instead
	$\rho_{i-1}<\rho_i<\rho_{i+1}$ or $\rho_{i-1}>\rho_i>\rho_{i+1}$ then $\Delta\sigma_i=0$ and the claim
	is still true.

	Finally, let us prove \eqref{eq:good-v-c}.
	If $\rho_i\leq c\leq\rho_{i+1}$ then $\sign(\rho_{i+1}-c)-\sign(\rho_i-c)\geq0$ and
	\[
		\left.
		\begin{aligned}
			U_i & \geq0 &  & \implies &  & v_i-v(c) = v(\rho_{i+1})-v(c) \leq 0 \\
			U_i & <0    &  & \implies &  & v_i-v(c) = v(\rho_i)-v(c) \geq 0
		\end{aligned}
		\right\} \implies \bigl(v_i-v(c)\bigr)U_i \leq 0;
	\]
	similarly, if $\rho_i\geq c\rho_{i+1}$, then $\sign(\rho_{i+1}-c)-\sign(\rho_i-c)\leq0$ and
	$\bigl(v_i-v(c)\bigr)U_i \geq 0$; so in both cases \eqref{eq:good-v-c} holds.
	If instead $c<\rho_i,\rho_{i+1}$ or $c>\rho_i,\rho_{i+1}$, then $\sign(\rho_{i+1}-c)-\sign(\rho_i-c)=0$
	and the inequality is trivially true.
\end{proof}

\begin{lemma}[First and second differences of the free velocity]\label{lem:diff-U}
	Let $V$ and $W$ satisfy the \autoref{ass},
	let $(x_0,\dots,x_N)$ be sorted particles and let $U_i$ be as defined in \eqref{eq:ode-U}. Then
	\begin{align}
		\label{eq:first-diff-U}
		\frac{\abs{U_i-U_{i-1}}}{x_i-x_{i-1}}
		          & \leq C_1+C_2\rho_i                          &
		\forall i & \in \{1,\dots,N\},                            \\
		\label{eq:second-diff-U}
		\abs*{\frac{U_{i+1}-U_i}{x_{i+1}-x_i} - \frac{U_i-U_{i-1}}{x_i-x_{i-1}}}
		          & \leq \Gamma_i + C_3 \abs{\rho_{i+1}-\rho_i} &
		\forall i & \in \{1,\dots,N-1\},
	\end{align}
	where
	\begin{subequations}
		\begin{align}
			C_1      & = \norm{\de_xV}_{L^\infty([x_0,x_N])}
			+ \norm{\de_x^2W}_{L^\infty([x_0-x_N,x_N-x_0])}\norm{\bar\rho}_{L^1(\setR)}, \\
			C_2      & = \abs{w},                                                        \\
			\Gamma_i & = \int_{x_{i-1}}^{x_{i+1}}
			\left( \abs{\de_x^2V(x)} + (\abs{\de_x^3W}*\bar\rho)(x) \right) \dx
			\spliteq +\norm{\de_x^2W}_{L^\infty([x_0-x_N,x_N-x_0])}
			[(q_i+q_{i+1})+(\rho_i+\rho_{i+1})(x_{i+1}-x_i)], \notag                     \\
			C_3      & = 2\norm{\de_xW}_{L^\infty([x_0-x_N,x_N-x_0])}.
		\end{align}
	\end{subequations}
	Moreover, assuming that \ref{as:W-repulsive} holds, i.e.\ $W$ is repulsive at small scales,
	then \eqref{eq:first-diff-U} can be refined with the one-sided estimate
	\begin{align}
		\label{eq:first-diff-U-repulsive}
		-\frac{U_i-U_{i-1}}{x_i-x_{i-1}} & \leq C_1 & \forall i & \in \{1,\dots,N\}.
	\end{align}
\end{lemma}

\begin{proof}
	Recall that $U=V-\de_xW*\bar\rho$; we will study the contributions of the two terms separately.
	By \ref{as:V} the field $V$ is Lipschitz, therefore
	\[
		\frac{\abs{V(x_i)-V(x_{i-1})}}{x_i-x_{i-1}} \leq \norm{\de_xV}_{L^\infty([x_0,x_N])}.
	\]
	On the other hand, by \ref{as:W} the spatial derivative of $\de_xW*\bar\rho$ is
	$D_x\de_xW*\bar\rho = \de_x^2W*\bar\rho+w\delta_0*\bar\rho = \de_x^2W*\bar\rho+w\bar\rho$, hence
	\[
		\begin{split}
			\frac{\abs{(\de_xW*\bar\rho)(x_i) - (\de_xW*\bar\rho)(x_{i-1})}}{x_i-x_{i-1}}
			&\leq \dashint_{x_{i-1}}^{x_i} \abs*{(\de_x^2W*\bar\rho)(x) + w\bar\rho(x)}\dx \\
			&\leq \norm{\de_x^2W*\bar\rho}_{L^\infty([x_0,x_N])} + \abs{w}\rho_i \\
			&\leq \norm{\de_x^2W}_{L^\infty([x_0-x_N,x_N-x_0])}\norm{\bar\rho}_{L^1(\setR)} + \abs{w}\rho_i.
		\end{split}
	\]
	Combining these two estimates proves \eqref{eq:first-diff-U}.
	The one-sided estimate \eqref{eq:first-diff-U-repulsive} is obtained by noticing that,
	under the assumption \ref{as:W-repulsive}, we have
	$\de_x(\de_xW*\bar\rho) = \de_x^2W*\bar\rho+w\bar\rho \leq \de_x^2W*\bar\rho$, hence
	\[
		\begin{split}
			\frac{(\de_xW*\bar\rho)(x_i) - (\de_xW*\bar\rho)(x_{i-1})}{x_i-x_{i-1}}
			&\leq \dashint_{x_{i-1}}^{x_i} (\de_x^2W*\bar\rho)(x)\dx \\
			&\leq \norm{\de_x^2W}_{L^\infty([x_0-x_N,x_N-x_0])}\norm{\bar\rho}_{L^1(\setR)}.
		\end{split}
	\]

	Let us now move on to \eqref{eq:second-diff-U}.
	By \ref{as:V} we have $V\in W^{2,1}((x_{i-1},x_{i+1}))$, so we can apply Taylor's theorem
	centered at $x_i$ with the remainder in integral form and obtain
	\[
		\begin{split}
			\abs*{\frac{V(x_{i+1})-V(x_i)}{x_{i+1}-x_i} - \frac{V(x_i)-V(x_{i-1})}{x_i-x_{i-1}}}
			&= \abs*{\int_{x_i}^{x_{i+1}} \de_x^2V(x) \frac{x_{i+1}-x}{x_{i+1}-x_i} \dx
			+ \int_{x_{i-1}}^{x_i} \de_x^2V(x) \frac{x-x_{i-1}}{x_i-x_{i-1}} \dx} \\
			&\leq \int_{x_{i-1}}^{x_{i+1}} \abs{\de_x^2V(x)}\dx.
		\end{split}
	\]
	To treat the interaction term, let us split $\bar\rho=\bar\rho_\text{near}+\bar\rho_\text{far}$
	with $\bar\rho_\text{near} = \bar\rho\rvert_{(x_{i-1},x_{i+1})}$
	and $\bar\rho_\text{far} = \bar\rho\rvert_{(x_{i-1},x_{i+1})^c}$.
	By \ref{as:W} we have $\de_xW*\bar\rho_\text{far}\in W^{2,1}\bigl((x_{i-1},x_{i+1})\bigr)$
	with $\de_x(\de_xW*\bar\rho_\text{far}) = \de_x^2W*\bar\rho_\text{far} + w\bar\rho_\text{far}
		= \de_x^2W*\bar\rho_\text{far}$
	and similarly $\de_x^2(\de_xW*\bar\rho_\text{far}) = \de_x^3W*\bar\rho_\text{far}$;
	therefore we can treat $\de_xW*\bar\rho_\text{far}$ as we did with $V$ and obtain
	\[
		\begin{split}
			&\abs*{\frac{(\de_xW*\bar\rho_\text{far})(x_{i+1})-(\de_xW*\bar\rho_\text{far})(x_i)}{x_{i+1}-x_i}
			- \frac{(\de_xW*\bar\rho_\text{far})(x_i)-(\de_xW*\bar\rho_\text{far})(x_{i-1})}{x_i-x_{i-1}}} \\
			&\leq \int_{x_{i-1}}^{x_{i+1}} \abs{(\de_x^3W*\bar\rho_\text{far})(x)}\dx
			\leq \int_{x_{i-1}}^{x_{i+1}} (\abs{\de_x^3W}*\bar\rho_\text{far})(x)\dx \\
			&\leq \int_{x_{i-1}}^{x_{i+1}} (\abs{\de_x^3W}*\bar\rho)(x)\dx.
		\end{split}
	\]
	On the other hand, the contribution of the near field is
	\[
		\begin{split}
			&\abs*{\frac{(\de_xW*\bar\rho_\text{near})(x_{i+1})-(\de_xW*\bar\rho_\text{near})(x_i)}{x_{i+1}-x_i}
			- \frac{(\de_xW*\bar\rho_\text{near})(x_i)-(\de_xW*\bar\rho_\text{near})(x_{i-1})}{x_i-x_{i-1}}} \\
			&= \abs*{\sum_{j=i}^{i+1} \int_{x_{j-1}}^{x_j} \left(
				\frac{\de_xW(x_{i+1}-x)-\de_xW(x_i-x)}{x_{i+1}-x_i}
				-\frac{\de_xW(x_i-x)-\de_xW(x_{i-1}-x)}{x_i-x_{i-1}}
				\right) \bar\rho(x) \dx} \\
			&\leq \int_{x_{i-1}}^{x_i} \abs*{\frac{\de_xW(x_{i+1}-x)-\de_xW(x_i-x)}{x_{i+1}-x_i}} \rho_i \dx
			+\int_{x_i}^{x_{i+1}} \abs*{\frac{\de_xW(x_i-x)-\de_xW(x_{i-1}-x)}{x_i-x_{i-1}}} \rho_{i+1}\dx
			\spliteq +\abs*{
			\int_{x_i}^{x_{i+1}} \frac{\de_xW(x_{i+1}-x)-\de_xW(x_i-x)}{x_{i+1}-x_i} \rho_{i+1} \dx
			-\int_{x_{i-1}}^{x_i} \frac{\de_xW(x_i-x)-\de_xW(x_{i-1}-x)}{x_i-x_{i-1}} \rho_i \dx}.
		\end{split}
	\]
	In the first integral $x_{i+1}-x,x_i-x>0$, whereas in the second integral $x_i-x,x_{i-1}-x<0$,
	so we can use the fact that, thanks to \ref{as:W}, $\de_xW$ is locally Lipschitz in $(-\infty,0]$
	and $[0,\infty)$ to bound both of them from above with
	\[
		\begin{split}
			&\int_{x_{i-1}}^{x_i} \abs*{\frac{\de_xW(x_{i+1}-x)-\de_xW(x_i-x)}{x_{i+1}-x_i}} \rho_i \dx
			+\int_{x_i}^{x_{i+1}} \abs*{\frac{\de_xW(x_i-x)-\de_xW(x_{i-1}-x)}{x_i-x_{i-1}}} \rho_{i+1}\dx \\
			&\leq \int_{x_{i-1}}^{x_i} \norm{\de_x^2W}_{L^\infty([0,x_N-x_0])} \rho_i \dx
			+\int_{x_i}^{x_{i+1}} \norm{\de_x^2W}_{L^\infty([x_0-x_N,0])} \rho_{i+1} \dx \\
			&\leq \norm{\de_x^2W}_{L^\infty([x_0-x_N,x_N-x_0])} (q_i+q_{i+1}).
		\end{split}
	\]
	To estimate the integrals in the last line instead, we do the following:
	we replace the densities with their average $\tilde\rho=(\rho_i+\rho_{i+1})/2$ by adding and subtracting
	the same quantities, then we use the denominators to interpret the integrals as averages and parametrize
	them in $[0,1]$ with linear changes of variables, $x=x_i+s(x_{i+1}-x_i)$ for the first one and
	$x=x_{i-1}+s(x_i-x_{i-1})$ for the second, and finally we mix the four terms to pair them more wisely.
	With these manipulations we get
	\[
		\begin{split}
			&\abs*{
			\int_{x_i}^{x_{i+1}} \frac{\de_xW(x_{i+1}-x)-\de_xW(x_i-x)}{x_{i+1}-x_i} \rho_{i+1} \dx
			-\int_{x_{i-1}}^{x_i} \frac{\de_xW(x_i-x)-\de_xW(x_{i-1}-x)}{x_i-x_{i-1}} \rho_i \dx} \\
			&\leq \tilde\rho \abs*{
			\dashint_{x_i}^{x_{i+1}} [\de_xW(x_{i+1}-x)-\de_xW(x_i-x)] \dx
			-\dashint_{x_{i-1}}^{x_i} [\de_xW(x_i-x)-\de_xW(x_{i-1}-x)] \dx}
			\spliteq +\abs{\rho_i-\tilde\rho}
			\dashint_{x_{i-1}}^{x_i} \abs*{\de_xW(x_i-x)-\de_xW(x_{i-1}-x)} \dx
			\spliteq +\abs{\rho_{i+1}-\tilde\rho}
			\int_{x_i}^{x_{i+1}} \abs*{\de_xW(x_{i+1}-x)-\de_xW(x_i-x)} \dx \\
			&\leq \tilde\rho \biggl\lvert \int_0^1 \bigl[
			\de_xW\bigl((1-s)(x_{i+1}-x_i)\bigr) -\de_xW\bigl(-s(x_{i+1}-x_i)\bigr)
			\spliteq\hspace{1.5em}
			-\de_xW\bigl((1-s)(x_i-x_{i-1})\bigr) +\de_xW\bigl(-s(x_i-x_{i-1})\bigr)
			\bigr] \ds \biggr\rvert
			\spliteq +2\norm{\de_xW}_{L^\infty([x_0-x_N,x_N-x_0])}
			(\abs{\rho_i-\tilde\rho}+\abs{\rho_{i+1}-\tilde\rho}) \\
			&\leq \tilde\rho \int_0^1
			\abs*{\de_xW\bigl((1-s)(x_{i+1}-x_i)\bigr)-\de_xW\bigl((1-s)(x_i-x_{i-1})\bigr)} \ds
			\spliteq +\tilde\rho \int_0^1
			\abs*{\de_xW\bigl(-s(x_i-x_{i-1})\bigr)-\de_xW\bigl(-s(x_{i+1}-x_i)\bigr)} \ds
			\spliteq +2\norm{\de_xW}_{L^\infty([x_0-x_N,x_N-x_0])} \abs{\rho_{i+1}-\rho_i},
		\end{split}
	\]
	and now the arguments to the $\de_xW(\plchldr)$'s that we are subtracting have the same sign,
	so we can use again the Lipschitz estimates\footnote{
	Notice that
	$\abs{(x_{i+1}-x_i)-(x_i-x_{i-1})}
		=\abs[\big]{\abs{x_{i+1}-x_i}-\abs{x_i-x_{i-1}}}
		\leq\abs{x_{i+1}-x_i}+\abs{x_i-x_{i-1}}
		=x_{i+1}-x_{i-1}$.}
	\[
		\begin{split}
			&\abs*{\de_xW\bigl((1-s)(x_{i+1}-x_i)\bigr)-\de_xW\bigl((1-s)(x_i-x_{i-1})\bigr)} \\
			&\leq \norm{\de_x^2W}_{L^\infty([x_0-x_N,x_N-x_0])} (1-s) \abs{(x_{i+1}-x_i)-(x_i-x_{i-1})} \\
			&\leq \norm{\de_x^2W}_{L^\infty([x_0-x_N,x_N-x_0])} (x_{i+1}-x_{i-1})
		\end{split}
	\]
	and
	\[
		\begin{split}
			&\abs*{\de_xW\bigl(-s(x_i-x_{i-1})\bigr)-\de_xW\bigl(-s(x_{i+1}-x_i)\bigr)} \\
			&\leq \norm{\de_x^2W}_{L^\infty([x_0-x_N,x_N-x_0])} s \abs{(x_{i+1}-x_i)-(x_i-x_{i-1})} \\
			&\leq \norm{\de_x^2W}_{L^\infty([x_0-x_N,x_N-x_0])} (x_{i+1}-x_{i-1})
		\end{split}
	\]
	to establish that
	\[
		\begin{split}
			&\abs*{
			\int_{x_i}^{x_{i+1}} \frac{\de_xW(x_{i+1}-x)-\de_xW(x_i-x)}{x_{i+1}-x_i} \rho_{i+1} \dx
			-\int_{x_{i-1}}^{x_i} \frac{\de_xW(x_i-x)-\de_xW(x_{i-1}-x)}{x_i-x_{i-1}} \rho_i \dx} \\
			&\leq 2\norm{\de_xW}_{L^\infty([x_0-x_N,x_N-x_0])} \abs{\rho_{i+1}-\rho_i}
			+2\tilde\rho \norm{\de_x^2W}_{L^\infty([x_0-x_N,x_N-x_0])} (x_{i+1}-x_{i-1}).
		\end{split}
	\]
	Putting together the estimates for $V$, $\de_xW*\bar\rho_\text{far}$ and
	$\de_xW*\bar\rho_\text{near}$ concludes the proof of \eqref{eq:second-diff-U}.
\end{proof}

The constants of the previous lemma will be refined in \autoref{rmk:diff-U-constants}.

\begin{proposition}[Uniform bound of the mass]\label{prop:bound-mass}
	Let $v$, $V$, $W$, $f$ satisfy the \autoref{ass} and let $\bar\rho$ be the piecewise constant density
	associated to particles $(x_0,\dots,x_N)$ and masses $(q_1,\dots,q_N)$ solving \eqref{eq:ode}.

	Then there exists an increasing function $Q \in C\bigl([0,\infty);[1,\infty)\bigr)$ independent of $N$
	such that
	\begin{align*}
		q_i(0)Q(t)^{-1} \leq & q_i(t) \leq q_i(0)Q(t), &  & \forall t\in[0,\infty),\ i\in\{1,\dots,N\}, \\
		q(0)Q(t)^{-1}   \leq & q(t)   \leq q(0)Q(t),   &  & \forall t\in[0,\infty).
	\end{align*}
\end{proposition}

\begin{proof}
	From the ODE \eqref{eq:ode-q} and \ref{as:f} we can estimate
	\[
		\abs{q_i'(t)}
		\leq \int_{x_{i-1}(t)}^{x_i(t)} \abs{f\bigl(t,x,\rho_i(t)\bigr)} \dx
		\leq \int_{x_{i-1}(t)}^{x_i(t)} c_f F(t) \rho_i(t) \dx
		= c_f F(t) q_i(t)
	\]
	therefore the statement is true with
	\[
		Q(t) = \exp\oleft(c_f \int_0^t F(\tau)\d\tau\right). \qedhere
	\]
\end{proof}

In the following \autoref{prop:bound-support} and \autoref{prop:bound-density} we make use of a
nonlinear generalization of Grönwall's inequality known as LaSalle--Bihari--Butler--Rogers' inequality,
see \cite{Bihari,Butler-Rogers}. A statement suitable for our purposes can also be found in
\cite[Theorem~2.4]{Radici-Stra-2023}.

\begin{proposition}[Uniform bound of the support]\label{prop:bound-support}
	Let $v$, $V$, $W$, $f$ satisfy the \autoref{ass} and let $\bar\rho$ be the piecewise constant density
	associated to particles $(x_0,\dots,x_N)$ and masses $(q_1,\dots,q_N)$ solving \eqref{eq:ode},
	with initial conditions $-S_0 \leq x_0(0) < x_N(0) \leq S_0$ for some $S_0>0$.

	Then there exists an increasing function $S \in C\bigl([0,\infty);[0,\infty)\bigr)$ independent of $N$
	such that
	\[
		-S(t) \leq x_0(t) \leq x_N(t) \leq S(t), \qquad \forall t\in[0,\infty),
	\]
	or equivalently $\supp\bigl(\bar\rho(t)\bigr) \subseteq [-S(t),S(t)]$.
\end{proposition}

\begin{proof}
	Following the proof of \cite[Proposition 2.5]{Radici-Stra-2023}, we find that the derivative of
	$s(t)=\max\{x_0(t)_-,x_N(t)_+\}$ can be estimated as
	\[
		s'(t) \leq \norm{v}_\infty F(t)[1+q(0)Q(t)] \lambda\bigl(s(t)\bigr),
	\]
	where the only difference is that the norm of the convolution is bounded by
	\[
		\norm{\de_x W*\bar\rho}_{L^\infty([-s(t),s(t)])}
		\leq \norm{\de_x W}_{L^\infty([-2s(t),2s(t)])}\norm{\bar\rho}_{L^1(\setR)}
		\leq q(0)Q(t)F(t)\lambda(s(t)).
	\]
	By the assumption \ref{as:mild-growth} on $\lambda$ we can apply Bihari's inequality and deduce that
	\[
		s(t) \leq S(t) = \Lambda^{-1}(\Lambda(S_0) + A(t))
	\]
	where $A$ and $\Lambda$ are primitives of $FQ$ and $1/\lambda$ respectively.
\end{proof}

\begin{remark}\label{rmk:diff-U-constants}
	When time dependence is taken into account, the constants appearing in \autoref{lem:diff-U} can be refined
	with the help of the \autoref{ass}, \autoref{prop:bound-mass} and \autoref{prop:bound-support}.
	Notice in particular that $C_1(t)$, $C_2(t)$ and $C_3(t)$ are uniform in $N$.
	\begin{align}
		C_1(t)      & \leq F(t) \bigl[G\bigl(S(t)\bigr)+G\bigl(2S(t)\bigr)q(0)Q(t)\bigr], \\
		C_2(t)      & \leq F(t),                                                          \\
		\Gamma_i(t) & \leq
		\int_{x_{i-1}(t)}^{x_{i+1}(t)} \left( \abs{\de_x^2V(t,x)} + (\abs{\de_x^3W}*\bar\rho)(t,x) \right) \dx
		\spliteq +F(t)G\bigl(2S(t)\bigr) \bigl[\bigl(q_i(t)+q_{i+1}(t)\bigr)
		+\bigl(\rho_i(t)+\rho_{i+1}(t)\bigr) \bigl(x_{i+1}(t)-x_i(t)\bigr)\bigr], \notag  \\
		C_3(t)      & \leq 2F(t)G\bigl(2S(t)\bigr).
	\end{align}
\end{remark}

\begin{proposition}[Uniform bound of the density]\label{prop:bound-density}
	Let $v$, $V$, $W$, $f$ satisfy the \autoref{ass} and let $\bar\rho$ be the piecewise constant density
	associated to particles $(x_0,\dots,x_N)$ and masses $(q_1,\dots,q_N)$ solving \eqref{eq:ode},
	with initial conditions $-S_0 \leq x_0(0) < x_N(0) \leq S_0$ and $\bar\rho(0)\leq R_0$ for some $S_0,R_0>0$.

	Then there exists an increasing function $R \in C\bigl([0,\infty);[0,\infty)\bigr)$ independent of $N$
	such that $\rho_i(t) \leq R(t)$ for every $t\in[0,\infty)$ and $i\in\{1,\dots,N\}$, or equivalently
	$\bar\rho(t)\leq R(t)$.
\end{proposition}

\begin{proof}
	Recall the formula \eqref{eq:ode-rho} for the derivative of $\rho_i(t)$.
	From \eqref{eq:ode-q} and \ref{as:f} we can estimate
	\[
		\abs{q_i'(t)}
		\leq \int_{x_{i-1}(t)}^{x_i(t)} \abs{f\bigl(t,x,\rho_i(t)\bigr)} \dx
		\leq \int_{x_{i-1}(t)}^{x_i(t)} c_f F(t) \rho_i(t) \dx
		= c_f F(t) q_i(t).
	\]
	For a fixed time $t$, let $i\in\{1,\dots,N\}$ be such that $\rho_i(t)=\max_j\rho_j(t)$;
	then by \eqref{eq:good-v-max} of \autoref{lem:good-v} we have that
	\[
		\begin{split}
			\rho_i'(t)
			&\leq -\rho_i(t)v\bigl(\rho_i(t)\bigr)\frac{U_i(t)-U_{i-1}(t)}{x_i(t)-x_{i-1}(t)}
			+ c_f F(t) \rho_i(t).
		\end{split}
	\]
	We now have to distinguish which of the two options of assumption \ref{as:no-collapse} holds.
	\begin{itemize}
		\item Assume that \ref{as:v-decays} holds, i.e.\ $v$ decays sufficiently fast.
		      By \eqref{eq:first-diff-U} of \autoref{lem:diff-U} we get
		      \[
			      \begin{split}
				      \rho_i'(t)
				      &\leq C_1(t) \rho_i(t)v\bigl(\rho_i(t)\bigr) + C_2(t) \rho_i(t)^2v\bigl(\rho_i(t)\bigr)
				      + c_fF(t)\rho_i(t) \\
				      &\leq
				      \bigl(C_1(t) \norm{v}_\infty + c_f F(t)\bigr) \rho_i(t)
				      + C_2(t) \rho_i(t)^2v\bigl(\rho_i(t)\bigr) \\
				      &\leq \bigl(C_1(t) \norm{v}_\infty + C_2(t) + c_f F(t)\bigr)
				      \bigl[\rho_i(t) + \rho_i(t)^2v\bigl(\rho_i(t)\bigr)\bigr] \\
				      &\leq \bigl(C_1(t) \norm{v}_\infty + C_2(t) + c_f F(t)\bigr) g\bigl(\rho_i(t)\bigr).
			      \end{split}
		      \]
		      Defining $r(t)=\max_i\rho_i(t)$ for every $t>0$, we find that
		      \[
			      r'(t) \leq \bigl(C_1(t) \norm{v}_\infty + C_2(t) + c_f F(t)\bigr) g\bigl(r(t)\bigr),
		      \]
		      to which we can apply Bihari's inequality.
		\item Assume that \ref{as:W-repulsive} holds, i.e.\ $W$ is repulsive at small scales.
		      By \eqref{eq:first-diff-U-repulsive} of \autoref{lem:diff-U} we get
		      \[
			      \rho_i'(t) \leq \bigl(C_1 \norm{v}_\infty + c_f F(t)\bigr) \rho_i(t).
		      \]
		      Defining $r(t)=\max_i\rho_i(t)$ for every $t>0$, we find that
		      \[
			      r'(t) \leq \bigl(C_1 \norm{v}_\infty + c_f F(t)\bigr) r(t),
		      \]
		      to which we can apply Grönwall's inequality.
	\end{itemize}
	In both cases we deduce the existence of an increasing function $R \in C\bigl([0,\infty);[0,\infty)\bigr)$
	such that $\max_j \rho_j(t) = r(t) \leq R(t)$.
\end{proof}

With the estimates of \autoref{prop:bound-mass}, \autoref{prop:bound-support} and
\autoref{prop:bound-density}, by a standard argument described more in detail in
\cite[Proposition~2.8]{Radici-Stra-2023}, we deduce the global existence and uniqueness of solutions
of \eqref{eq:ode}.

\begin{proposition}[Uniform bound of the total variation]\label{prop:bound-TV}
	Let $v$, $V$, $W$, $f$ satisfy the \autoref{ass} and let $\bar\rho$ be the piecewise constant density
	associated to particles $(x_0,\dots,x_N)$ and masses $(q_1,\dots,q_N)$ solving \eqref{eq:ode},
	with initial conditions $-S_0 \leq x_0(0) < x_N(0) \leq S_0$, $\bar\rho(0)\leq R_0$ and
	$\TV\bigl(\bar\rho(0)\bigr)\leq B_0$ for some $S_0,R_0,B_0>0$.

	Then there exists an increasing function $B \in C\bigl([0,\infty);[0,\infty)\bigr)$ independent of $N$
	such that $\TV\bigl(\rho(t)\bigr) \leq B(t)$ for every $t\in[0,\infty)$.
\end{proposition}

\begin{proof}
	Let $\sigma_i=\sign(\rho_{i+1}-\rho_i)$, $\Delta\sigma_i=\sigma_i-\sigma_{i-1}$, $\Delta U_i=U_i-U_{i-1}$,
	$\Delta x_i=x_i-x_{i-1}$.

	Using \eqref{eq:ode-rho} and summing by parts the contribution arising from the advection we can compute
	\[
		\begin{split}
			\frac{\d}{\dt} \TV\bigl(\bar\rho(t)\bigr)
			&= \sum_{i=0}^N \sigma_i (\rho_{i+1}'-\rho_i')
			= \sum_{i=0}^N \sigma_i (\rho_{i+1}'^{\mathrm{adv}}-\rho_i'^{\mathrm{adv}})
			+\sum_{i=0}^N \sigma_i (\rho_{i+1}'^{\mathrm{src}}-\rho_i'^{\mathrm{src}}) \\
			&= -\sum_{i=1}^N \Delta\sigma_i \rho_i'^{\mathrm{adv}}
			+\sum_{i=0}^N \sigma_i (\rho_{i+1}'^{\mathrm{src}}-\rho_i'^{\mathrm{src}}) \\
			&= \sum_{i=1}^N \Delta\sigma_i \rho_i \frac{x_i'-x_{i-1}'}{x_i-x_{i-1}}
			+\sum_{i=0}^N \sigma_i (\rho_{i+1}'^{\mathrm{src}}-\rho_i'^{\mathrm{src}}).
		\end{split}
	\]
	Using \eqref{eq:good-v-step} of \autoref{lem:good-v} on the addends of the first sum we can estimate
	$\frac{\d}{\dt} \TV\bigl(\bar\rho(t)\bigr) \leq \Sigma_\mathrm{adv} + \Sigma_\mathrm{src}$ where
	\begin{align*}
		\Sigma_\mathrm{adv} & = \sum_{i=1}^N \Delta\sigma_i \rho_i v(\rho_i) \frac{\Delta U_i}{\Delta x_i}, \\
		\Sigma_\mathrm{src} & =
		\sum_{i=1}^{N-1} \sigma_i \left(\dashint_{x_i}^{x_{i+1}} f(t,x,\rho_{i+1})\dx
		-\dashint_{x_{i-1}}^{x_i} f(t,x,\rho_i)\dx\right)
		\spliteq +\sigma_0 \dashint_{x_0}^{x_1} f(t,x,\rho_1)\dx
		-\sigma_N \dashint_{x_{N-1}}^{x_N} f(t,x,\rho_N)\dx.
	\end{align*}
	Using \autoref{lem:diff-U} with the constants given by \autoref{rmk:diff-U-constants} we can now bound
	\[
		\begin{split}
			\Sigma_\mathrm{adv}
			&= -\sum_{i=1}^{N-1} \sigma_i \left( m(\rho_{i+1})\frac{\Delta U_{i+1}}{\Delta x_{i+1}}
			-m(\rho_i)\frac{\Delta U_i}{\Delta x_i} \right)
			+\sigma_N m(\rho_N) \frac{\Delta U_N}{\Delta x_N} -\sigma_0 m(\rho_1) \frac{\Delta U_1}{\Delta x_1} \\
			&= -\sum_{i=1}^{N-1} \sigma_i m(\rho_{i+1})
			\left( \frac{\Delta U_{i+1}}{\Delta x_{i+1}} - \frac{\Delta U_i}{\Delta x_i} \right)
			-\sum_{i=1}^{N-1} \sigma_i \bigl(m(\rho_{i+1})-m(\rho_i)\bigr) \frac{\Delta U_i}{\Delta x_i}
			\spliteq+\sigma_N m(\rho_N) \frac{\Delta U_N}{\Delta x_N} -\sigma_0 m(\rho_1) \frac{\Delta U_1}{\Delta x_1} \\
			&\leq \sum_{i=1}^{N-1} m(\rho_{i+1})
			\abs*{ \frac{\Delta U_{i+1}}{\Delta x_{i+1}} - \frac{\Delta U_i}{\Delta x_i} }
			+\sum_{i=1}^N \abs{m(\rho_{i+1})-m(\rho_i)} \abs*{\frac{\Delta U_i}{\Delta x_i}}
			+m(\rho_1) \abs*{\frac{\Delta U_1}{\Delta x_1}} \\
			&\leq \norm{m}_{L^\infty([0,R(t)])} \sum_{i=1}^{N-1} (\Gamma_i+C_3\abs{\rho_{i+1}-\rho_i})
			+\norm{m'}_{L^\infty([0,R(t)])} \sum_{i=1}^N \abs{\rho_{i+1}-\rho_i} \bigl(C_1+C_2\rho_i\bigr)
			\spliteq+\norm{m}_{L^\infty([0,R(t)])} \bigl(C_1+C_2\rho_1\bigr) \\
			&\leq \norm{v}_\infty R(t) \sum_{i=1}^{N-1} \biggl\{
			\int_{x_{i-1}(t)}^{x_{i+1}(t)} \left( \abs{\de_x^2V(t,x)} + (\abs{\de_x^3W}*\bar\rho)(t,x) \right) \dx
			\spliteq\qquad +F(t)G\bigl(2S(t)\bigr) \bigl[\bigl(q_i(t)+q_{i+1}(t)\bigr)
			+\bigl(\rho_i(t)+\rho_{i+1}(t)\bigr) \bigl(x_{i+1}(t)-x_i(t)\bigr)\bigr]
			\spliteq\qquad +2F(t)G\bigl(2S(t)\bigr)\abs{\rho_{i+1}-\rho_i} \biggr\}
			\spliteq +\bigl[\norm{v}_\infty+R(t)G\bigl(R(t)\bigr)\bigr]
			F(t)\bigl[G\bigl(S(t)\bigr)+G\bigl(2S(t)\bigr)q(0)Q(t)+R(t)\bigr]
			\sum_{i=1}^N \abs{\rho_{i+1}-\rho_i}
			\spliteq +\norm{v}_\infty R(t) F(t)\bigl[G\bigl(S(t)\bigr)+G\bigl(2S(t)\bigr)q(0)Q(t)+R(t)\bigr] \\
			&\leq \norm{v}_\infty R(t) \left(\norm{\de_x^2V}_{L^1([-S(t),S(t)])}
			+ \norm{\abs{\de_x^3W}*\bar\rho}_{L^1([-S(t),S(t)])}\right)
			\spliteq +2\norm{v}_\infty R(t) F(t) G\bigl(2S(t)\bigr) \sum_{i=1}^N \left[
			q_i(t) + R(t)(x_i-x_{i-1}) + \abs{\rho_{i+1}-\rho_i}\right]
			\spliteq +\bigl[\norm{v}_\infty+R(t)G\bigl(R(t)\bigr)\bigr]
			F(t)\bigl[G\bigl(S(t)\bigr)+G\bigl(2S(t)\bigr)q(0)Q(t)+R(t)\bigr]
			\sum_{i=1}^N \abs{\rho_{i+1}-\rho_i}
			\spliteq +\norm{v}_\infty R(t) F(t)\bigl[G\bigl(S(t)\bigr)+G\bigl(2S(t)\bigr)q(0)Q(t)+R(t)\bigr] \\
			&\leq \norm{v}_\infty R(t) F(t)\bigl[G\bigl(S(t)\bigr)+G\bigl(2S(t)\bigr)q(0)Q(t)\bigr]
			\spliteq +2\norm{v}_\infty R(t) F(t) G\bigl(2S(t)\bigr) \bigl[q(0)Q(t) + R(t)S(t)\bigr]
			\spliteq +4\bigl[\norm{v}_\infty+R(t)G\bigl(R(t)\bigr)\bigr] R(t)
			F(t)\bigl[G\bigl(S(t)\bigr)+G\bigl(2S(t)\bigr)q(0)Q(t)\bigr] \TV(\bar\rho)
		\end{split}
	\]
	On the other hand, by, assumption \ref{as:f}, we have
	\[
		\begin{split}
			\Sigma_\mathrm{src}
			&\leq \sum_{i=1}^{N-1} \abs*{\dashint_{x_i}^{x_{i+1}} f(t,x,\rho_{i+1})\dx
			-\dashint_{x_{i-1}}^{x_i} f(t,x,\rho_i)\dx}
			\spliteq
			+\abs*{\dashint_{x_0}^{x_1} f(t,x,\rho_1)\dx}
			+\abs*{\dashint_{x_{N-1}}^{x_N} f(t,x,\rho_N)\dx} \\
			&\leq \sum_{i=1}^{N-1}
			\dashint_{x_i}^{x_{i+1}} \abs[\big]{f(t,x,\rho_{i+1}) - f(t,x,\rho_i)}\dx
			\spliteq +\sum_{i=1}^{N-1}
			\abs*{\dashint_{x_i}^{x_{i+1}} f(t,x,\rho_i)\dx - \dashint_{x_{i-1}}^{x_i} f(t,x,\rho_i)\dx}
			\spliteq +\abs*{\dashint_{x_0}^{x_1} f(t,x,\rho_1)\dx}
			+\abs*{\dashint_{x_{N-1}}^{x_N} f(t,x,\rho_N)\dx} \\
			&\leq \sum_{i=1}^{N-1} \norm{\de_\rho f}_{L^\infty([0,R(t)])}\abs{\rho_{i+1}-\rho_i}
			+c_fF(t)(\rho_1+\rho_N)
			\spliteq +\sum_{i=1}^{N-1} \int_0^1
			\abs*{ f\bigl(t,(1-s)x_i+sx_{i+1},\rho_i\bigr) - f\bigl(t,(1-s)x_{i-1}+sx_i,\rho_i\bigr) } \ds \\
			&\leq F(t)G\bigl(R(t)\bigr) \TV(\bar\rho) + 2c_fF(t)R(t)
			+\sum_{i=1}^{N-1} \abs{D_x f(t,\plchldr,\rho_i)}([x_{i-1},x_{i+1}])\ds \\
			&\leq F(t)G\bigl(R(t)\bigr) \TV(\bar\rho) + 2c_fF(t)R(t)
			+2\eta_{t,R(t)}([-S(t),S(t)]).
		\end{split}
	\]
	Combining the estimates for $\Sigma_\mathrm{adv}$ and $\Sigma_\mathrm{src}$ we obtain that
	$\frac{\d}{\dt}\TV(\bar\rho) \leq \alpha(t) + \beta(t)\TV(\bar\rho)$ for some positive increasing
	functions $\alpha,\beta$ independent of $N$, hence the conclusion is reached by applying Grönwall's
	inequality.
\end{proof}

\subsection{Compactness}\label{sub:compactness}

This subsection is devoted to showing that the numerical approximations $\bar\rho^N$ converge (up to
subsequence) to a limit density $\rho$ in the strong topology of
$L^1_\loc\bigl([0,\infty)\times\setR\bigr)$. The main tool to prove the compactness of $(\bar\rho^N)_N$
is a weaker version of \cite[Theorem~2]{Rossi-Savare-2003}.

\renewcommand{\tmpcite}{\cite[Theorem~2]{Rossi-Savare-2003}}
\begin{theorem}[Weaker version of \tmpcite]\label{thm:rossi-savare}
	Let $X$ be a separable Banach space, let $\G:X\to[0,\infty]$ be a lower semi-continuous coercive
	functional, let $g:X\times X\to[0,\infty]$ be a lower semi-continuous generalized\,\footnote{The term
		``generalized'' refers to the fact that the function $g$ can assume the value $\infty$.}\! distance and
	let $\mathcal U$ be a family of measurable functions $u:(0,T)\to X$. Assume that
	\begin{subequations}
		\begin{equation}\label{eq:rossi-savare-hyp1}
			\sup_{u\in\mathcal{U}} \int_0^T \G\bigl(u(t)\bigr) \dt < \infty
		\end{equation}
		and
		\begin{equation}\label{eq:rossi-savare-hyp2}
			\lim_{h\to0} \sup_{u\in\mathcal{U}} \int_0^{T-h} g\bigl(u(t),u(t+h)\bigr) \dt = 0.
		\end{equation}
	\end{subequations}
	Then $\mathcal{U}$ contains a sequence $(u_n)_{n\in\setN}$ which converges in measure to a limit function
	$u:(0,T)\to X$, i.e.\ for every $\eps>0$ one has
	\[
		\lim_{n\to\infty} \leb^1\bigl(\set{t\in(0,T)}{\norm{u_n(t)-u(t)}_X > \eps}\bigr) = 0.
	\]
\end{theorem}

We will apply this theorem to the Banach space $\bigl(\Meas(\setR),\norm{\plchldr}_1\bigr)$, where
$\Meas(\setR)$ denotes the set of (finite) signed Borel measures and $\norm{\mu}_1=\abs{\mu}(\setR)$ is
the total variation norm, and $\mathcal{U}=\set{t\mapsto\bar\rho^N(t,\plchldr)}{N\in\setN_+}$.

Condition \eqref{eq:rossi-savare-hyp2} is implied by the equicontinuiy of $\mathcal{U}$ with respect to
the distance $g$. In \cite{Radici-Stra-2023} we used the Wasserstein distance $g=W_1$ because the $\bar\rho^N$ were probability measures for every time and the $L^\infty$ bound of the velocities of the particles implied the equicontinuity in this topology.
In our current setting the source term causes the approximate solutions $\bar\rho^N$ to (possibly) have
mass different from $1$, and moreover its contribution is controlled in the $\norm{\plchldr}_1$ topology
and not the weaker $W_1$.
For this reason we are led to consider a generalized distance $\g$,
introduced in \autoref{def:g} and studied in \autoref{lem:g}, which is dominated by both
$\norm{\plchldr}_1$ and $W_1$.

\begin{definition}\label{def:g}
	Define $\V,\W,\g:\Meas(\setR)\times\Meas(\setR)\to[0,\infty]$ as
	\begin{align*}
		\V(\mu,\nu) & = \sup\set*{\abs*{\int_\setR f\d\mu - \int_\setR f\d\nu}}
		{f\in C_b(\setR),\ \norm{f}_\infty \leq 1},                             \\
		\W(\mu,\nu) & = \sup\set*{\abs*{\int_\setR f\d\mu - \int_\setR f\d\nu}}
		{f\in \Lip_b(\setR),\ \Lip(f) \leq 1},                                  \\
		\g(\mu,\nu) & = \inf\set*{\sum_{i=1}^n d_i(\sigma_{i-1},\sigma_i)}
		{\sigma_i\in\Meas(\setR),\ \sigma_0=\mu,\ \sigma_n=\nu,\ d_i\in\{\mathcal{V},\mathcal{W}\}}.
	\end{align*}
\end{definition}

It is well known that $\V(\mu,\nu)=\norm{\mu-\nu}_1$ for all $\mu,\nu\in\Meas(\setR)$; on the other hand, by
the duality formula of optimal transport, when $\mu,\nu\in\Meas_+(\setR)$ are non-negative and
$\mu(\setR)=\nu(\setR)$, then $\W$ coincides with the Wasserstein distance $W_1$
\[
	\W(\mu,\nu) = W_1(\mu,\nu)
	= \inf\set*{\int_{\setR\times\setR} \abs{y-x}\d\pi(x,y)}
	{\pi\in\Meas_+(\setR),\ \proj^1_\#\pi=\mu,\ \proj^2_\#\pi=\nu}.
\]

\begin{lemma}\label{lem:g}
	The function $\g$ defined in \autoref{def:g} is a continuous generalized distance on $\Meas(\setR)$.
	Moreover,
	\begin{itemize}
		\item $\g$ locally metrizes the weak convergence of measures $\Meas(\setR)$ in duality with $C_b(\setR)$,
		      i.e.\ if $(\mu_n)_{n\in\setN}\subset\Meas(\setR)$ and $\mu\in\Meas(\setR)$ are such that
		      $\sup_{n\in\setN}\norm{\mu_n}_1<\infty$ and $\lim_{n\to\infty}\g(\mu_n,\mu)=0$ then
		      $\mu_n\weakto\mu$ in duality with $C_b(\setR)$;
		\item $\g$ metrizes the weak convergence of non-negative measures $\Meas_+(\setR)$ in duality with
		      $C_b(\setR)$,
		      i.e.\ if $(\mu_n)_{n\in\setN}\subset\Meas_+(\setR)$ and $\mu\in\Meas_+(\setR)$ are such that
		      $\lim_{n\to\infty}\g(\mu_n,\mu)=0$ then $\mu_n\weakto\mu$ in duality with $C_b(\setR)$.
	\end{itemize}
\end{lemma}

\begin{proof}
	First of all notice that $\V$ and $\W$ are symmetric and non-negative, hence $\g$ inherits these same
	properties. Moreover $\g(\mu,\mu)=0$ by definition and the triangle inequality for $\g$ is
	an immediate consequence of the non-negativity of $\V$ and $\W$. Let us now prove that $\g(\mu,\nu)=0$
	implies $\mu=\nu$.
	By definition of $\V$ and $\W$, for every $f\in\Lip_b(\setR)$ we have
	\begin{align*}
		\abs*{\int f\d\mu - \int f\d\nu} & \leq \norm{f}_\infty \V(\mu,\nu), &
		\abs*{\int f\d\mu - \int f\d\nu} & \leq \Lip(f) \W(\mu,\nu),
	\end{align*}
	hence
	\[
		\abs*{\int f\d\mu - \int f\d\nu} \leq \max\{\norm{f}_\infty,\Lip(f)\} \min\{\V(\mu,\nu),\W(\mu,\nu)\}.
	\]
	By definition of $\g$	we can then compute
	\[
		\begin{split}
			\abs*{\int f\d\mu - \int f\d\nu}
			&= \abs*{\sum_{i=1}^n \left(\int f\d\sigma_i - \int f\d\sigma_{i-1}\right)}
			\leq \sum_{i=1}^n \abs*{\int f\d\sigma_i - \int f\d\sigma_{i-1}} \\
			&\leq \sum_{i=1}^n \max\{\norm{f}_\infty,\Lip(f)\} d_i(\sigma_{i-1},\sigma_i) \\
			&= \max\{\norm{f}_\infty,\Lip(f)\} \sum_{i=1}^n d_i(\sigma_{i-1},\sigma_i).
		\end{split}
	\]
	Taking the infimum yields
	\[
		\abs*{\int f\d\mu - \int f\d\nu} \leq \max\{\norm{f}_\infty,\Lip(f)\} \g(\mu,\nu).
	\]
	If $\mu\neq\nu$, then there exists $f\in C^\infty_c(\setR) \subset \Lip_b(\setR)$ such that
	$\int f\d\mu \neq \int f\d\nu$, therefore $\g(\mu,\nu)>0$.

	The continuity of $\g$ with respect to the strong topology of $\Meas(\setR)$ follows from the triangle
	inequality and the fact that $\g(\mu,\nu) \leq \V(\mu,\nu)=\norm{\mu-\nu}_1$.

	Let us now prove that the convergence with respect to $\g$ implies the weak convergence.
	Let $(\mu_n)_{n\in\setN}$ be a sequence such that $\g(\mu_n,\mu)\to0$ and $\sup_n\norm{\mu_n}_1<\infty$.
	Fix $\phi\in C_b(\setR)$.
	Given $\eps>0$ there is $f\in\Lip_b(\setR)$ such that $\norm{f-\phi}_\infty\leq\eps$.
	With this we can compute
	\[
		\begin{split}
			\abs*{\int \phi \d\mu_n - \int \phi \d\mu}
			&\leq
			\abs*{\int f \d\mu_n - \int f \d\mu} + \norm{f-\phi}_\infty (\norm{\mu_n}_1+\norm{\mu}_1) \\
			&\leq \max\{\norm{f}_\infty,\Lip(f)\} \g(\mu_n,\mu) + \eps (\norm{\mu_n}_1+\norm{\mu}_1),
		\end{split}
	\]
	therefore
	\[
		\limsup_{n\to\infty} \abs*{\int \phi \d\mu_n - \int \phi \d\mu}
		\leq \eps \left(\sup_n\norm{\mu_n}_1+\norm{\mu}_1\right)
	\]
	and from the arbitrariness of $\eps>0$ we conclude $\int f\d\mu_n\to\int f\d\mu$.

	In the case of non-negative measures $\mu_n,\mu\in\Meas_+(\setR)$, we have that $\norm{\mu_n}_1=\int
		1\d\mu_n \to \int1\d\mu=\norm{\mu}_1$ because the constant function $1$ is Lipschitz and bounded, so
	the hypothesis of the uniform bound of the masses is verified and we can apply the previous result.
\end{proof}

With the generalized distance $\g$ at our disposal, we are now in position to prove the compactness of
the numerical approximations.

\begin{proposition}[Compactness]\label{prop:compactness}
	Let $v$, $V$, $W$, $f$ satisfy the \autoref{ass}.
	For every $N\in\setN_+$ let $\bar\rho^N$ be the piecewise constant density
	associated to particles $(x_0,\dots,x_N)$ and masses $(q_1,\dots,q_N)$ solving \eqref{eq:ode},
	with initial condition given by \autoref{lem:initial-datum}.

	Then there exists a non-negative function $\rho\in L^\infty_\loc([0,\infty)\times\setR)$ such that
	up to a subsequence $\bar\rho^N$ converges to $\rho$ strongly in $L^p([0,T]\times\setR)$ for every $T>0$
	and every $p\in[1,\infty)$.

	Moreover $\rho$ enjoys the estimates
	\begin{align*}
		\norm{\rho(t,\plchldr)}_{L^1(\setR)} & \leq \norm{\rho_0}_1 Q(t), &
		\rho(t,\plchldr)                     & \leq R(t),                   \\
		\supp \rho(t,\plchldr)               & \subseteq [-S(t),S(t)],    &
		\TV\bigl(\rho(t,\plchldr)\bigr)      & \leq B(t),
	\end{align*}
	for every $t\in[0,\infty)$ for some increasing functions $Q,R,S,B:[0,\infty)\to[0,\infty)$.
\end{proposition}

\begin{proof}
	Let $(X,\norm{\plchldr}_X) = (\Meas(\setR),\norm{\plchldr}_1)$ and $\U = \set{\bar\rho^N}{N\in\setN_+}$.
	Consider the functional $\G:X\to[0,\infty]$
	\[
		\G(\rho) = \sup_{x\in\supp(\rho)} \abs{x} + \norm{\rho}_{BV}
		= \sup_{x\in\supp(\rho)} \abs{x} + \norm{\rho}_1 + \TV(\rho).
	\]
	$\G$ is lower semi-continuous and coercive; moreover by \autoref{prop:bound-support},
	\autoref{prop:bound-mass} and \autoref{prop:bound-TV} we have that \eqref{eq:rossi-savare-hyp1}
	is satisfied.

	In order to verify \eqref{eq:rossi-savare-hyp2} with $g=\g$, we now prove that the functions
	$t\mapsto\bar\rho^N(t,\plchldr)$ are equi-continuous with respect to $\g$.

	Let $\Xi^N:\setR\to\setR$ be the increasing piecewise-affine function that maps every interval $[x_{i-1}(t),x_i(t)]$ to $[x_{i-1}(t+h),x_i(t+h)]$:
	\[
		\Xi^N(x) = \sum_{i=1}^N \left(
		\frac{x-x^N_{i-1}(t)}{x^N_i(t)-x^N_{i-1}(t)} x^N_i(t+h)
		+\frac{x^N_i(t)-x}{x^N_i(t)-x^N_{i-1}(t)} x^N_{i-1}(t+h)
		\right) \bm1_{[x^N_{i-1}(t),x^N_i(t)]}(x).
	\]
	By construction we have that
	\[
		\abs{\Xi^N(x)-x} \leq \abs{x_{i-1}(t+h)-x_{i-1}(t)} + \abs{x_i(t+h)-x_i(t)}
		\qquad\forall x\in[x_{i-1}(t),x_i(t)],
	\]
	therefore
	\[
		\begin{split}
			W_1\bigl(\bar\rho^N(t),\Xi^N_\#\bar\rho^N(t)\bigr)
			&\leq \int_\setR \abs{\Xi^N(x)-x} \d\bar\rho^N(x) \\
			&\leq \sum_{i=1}^N \int_{x_{i-1}(t)}^{x_i(t)} \abs{\Xi^N(x)-x} \d\bar\rho^N(x) \\
			&\leq \sum_{i=1}^N q_i(t) \bigl( \abs{x_{i-1}(t+h)-x_{i-1}(t)} + \abs{x_i(t+h)-x_i(t)} \bigr) \\
			&\leq 2 q(0)Q(t) \left(\max_{i\in\{1,\dots,N\}} \int_t^{t+h} \abs{x_i'(\tau)} \d\tau\right).
		\end{split}
	\]
	Notice that by \eqref{eq:ode-x}, \eqref{eq:ode-U}, and \ref{as:v}--\ref{as:W}
	for every $\tau\in(t,t+h)$ and $i\in\{1,\dots,N\}$ we have
	\[
		\begin{split}
			\abs{x_i'(\tau)}
			&\leq \norm{v}_\infty \bigl( \norm{V}_{L^\infty(\supp\bar\rho^N)}
			+ \norm{\de_xW*\bar\rho^N}_{L^\infty(\supp\bar\rho^N)} \bigr) \\
			&\leq \norm{v}_\infty \bigl[ F(t+h)G\bigl(S(t+h)\bigr)
				+ q(0)Q(t+h) F(t+h)G\bigl(2S(t+h)\bigr) \bigr],
		\end{split}
	\]
	therefore
	\begin{equation}\label{eq:equicont-adv}
		\begin{split}
			&W_1\bigl(\bar\rho^N(t),\Xi^N_\#\bar\rho^N(t)\bigr) \\
			&\leq 2 \norm{v}_\infty q(0)Q(t)F(t+h)
			\bigl[ G\bigl(S(t+h)\bigr) + q(0)Q(t+h) G\bigl(2S(t+h)\bigr) \bigr]
			h.
		\end{split}
	\end{equation}
	On the other hand
	\[
		\begin{split}
			\bar\rho^N(t+h) &= \sum_{i=1}^N \frac{q_i(t+h)}{x_i(t+h)-x_{i-1}(t+h)} \bm1_{(x_{i-1}(t+h),x_i(t+h))} \\
			\Xi^N_\#\bar\rho^N(t) &= \sum_{i=1}^N \frac{q_i(t)}{x_i(t+h)-x_{i-1}(t+h)} \bm1_{(x_{i-1}(t+h),x_i(t+h))}
		\end{split}
	\]
	hence by \eqref{eq:ode-q} and \ref{as:f}
	\begin{equation}\label{eq:equicont-source}
		\begin{split}
			\norm[\big]{\Xi^N_\#\bar\rho^N(t)-\bar\rho^N(t+h)}_1
			&= \sum_{i=1}^n \abs{q_i(t+h)-q_i(t)}
			\leq \sum_{i=1}^n \int_t^{t+h} \abs{q_i'(\tau)} \d\tau \\
			&\leq \sum_{i=1}^n \int_t^{t+h} c_f F(\tau) q_i(\tau) \d\tau
			\leq c_f F(t+h) q(0)Q(t+h) h.
		\end{split}
	\end{equation}
	Combining \eqref{eq:equicont-adv} and \eqref{eq:equicont-source}, for $t\in(0,T-h)$ we can finally  get the uniform estimate
	\[
		\begin{split}
			&\g\bigl(\bar\rho^N(t),\bar\rho^N(t+h)\bigr) \\
			&\leq W_1\bigl(\bar\rho^N(t),\Xi^N_\#\bar\rho^N(t)\bigr)
			+ \norm[\big]{\Xi^N_\#\bar\rho^N(t)-\bar\rho^N(t+h)}_1 \\
			&\leq q(0)Q(T)F(T) \bigl\{2 \norm{v}_\infty
			\bigl[ G\bigl(S(T)\bigr) + q(0)Q(T) G\bigl(2S(T)\bigr) \bigr]
			+ c_f \bigr\} h,
		\end{split}
	\]
	from which we deduce the validity of \eqref{eq:rossi-savare-hyp2}.
	We can now apply \autoref{thm:rossi-savare} and obtain that there exists a curve of measures
	$\rho:[0,\infty)\to\Meas(\setR)$ such that (up to a subsequence)
	$(t\mapsto\bar\rho^N(t,\plchldr))_{N\in\setN_+}$ converges in measure and for almost every time
	to $t\mapsto\rho(t,\plchldr)$.

	Since in the target space $\Meas(\setR)$ the topology considered is the one of strong convergence of
	measures, the limit $\rho$ is non-negative and all the bounds provided by \autoref{prop:bound-mass},
	\autoref{prop:bound-support}, \autoref{prop:bound-density} and \autoref{prop:bound-TV} pass to the limit,
	thus providing the estimates claimed in the statement.

	From the convergence for almost every time and the domination
	$\norm{\bar\rho^N}_1 \leq \sup_N q^N(0)Q(t)$ we deduce the convergence in
	$L^1\bigl([0,T];L^1(\setR)\bigr)=L^1\bigl([0,T]\times\setR\bigr)$.
	In addition, since $\bar\rho^N,\rho\leq R(T)$, for every $p\in[1,\infty)$ we have that
	\[
		\begin{split}
			\norm{\bar\rho^N-\rho}_{L^p([0,T]\times\setR)}^p
			&\leq \norm{\bar\rho^N-\rho}_{L^\infty([0,T]\times\setR)}^{p-1}
			\norm{\bar\rho^N-\rho}_{L^1([0,T]\times\setR)} \\
			&\leq [2R(T)]^{p-1} \norm{\bar\rho^N-\rho}_{L^1([0,T]\times\setR)},
		\end{split}
	\]
	hence there is also strong convergence in $L^p\bigl([0,T]\times\setR\bigr)$.
	Similarly, from the convergence in $L^1(\setR)$ for almost every time and the bound in $L^\infty(\setR)$
	we get that for almost every $t\in[0,T]$ we have $\bar\rho^N(t,\plchldr)\to\rho(t,\plchldr)$
	in $L^p(\setR)$.
\end{proof}

As a consequence of the previous theorem, one can prove the next corollary by following the argument used
for \cite[Corollary~2.12]{Radici-Stra-2023} with minimal modifications (in particular the fact that
$\rho$ is weakly continuous in time because $\bar\rho^N$ are equi-Lipschitz with respect to $\g$).

\begin{corollary}\label{cor:continuity}
	With the assumptions of \autoref{prop:compactness}, we have that $\rho\in C\bigl([0,\infty);L^p(\setR)\bigr)$ for every $p\in[1,\infty)$.
\end{corollary}

\subsection{Convergence to entropy solutions}\label{sub:entropy}

In the present subsection, we show in \autoref{prop:quasi-entropy-solution} that the numerical
approximations $\bar\rho^N$ are almost entropy solutions of the PDE \eqref{eq:pde} in a sense which is
detailed by \autoref{def:quasi-entropy-solution}, which is adapted from \cite{Marconi-Radici-Stra-2022}.
In \autoref{prop:entropy-solution} we show that in the limit $N\to\infty$ the quasi-entropy inequalities
converge thus providing that the limit $\rho$ obtained from \autoref{prop:compactness} is an entropy
solution according to \autoref{def:entropy-solution}. All the pieces are put together at the end of the
section where we finally prove the main \autoref{thm:main}.

\begin{definition}[Quasi-entropy solution]\label{def:quasi-entropy-solution}
	Let $v,V,W,f$ satisfy the \autoref{ass}.
	Let $\mu_{0,t},\mu_{1,t},\nu_{0,t},\nu_{1,t}\in L^1_\loc\bigl([0,\infty);\Meas(\setR)_+\bigr)$ be
	locally-finite non-negative Borel measures.

	We say that a non-negative function
	$\rho\in C\bigl([0,T);L^1_\loc(\setR)\bigr)\cap L^\infty_\loc\bigl([0,T);BV(\setR)\bigr)$
	is a $(\mu_0,\mu_1)$--quasi-entropy solution of \eqref{eq:pde} if, letting $U=V-\de_xW*\rho$, the
	following entropy inequality
	\begin{equation*}
		\begin{split}
			&\int_0^T\!\!\int_\setR \Bigl\{
			\abs{\rho-c}\de_t\phi + \sign(\rho-c)\bigl[\bigl(m(\rho)-m(c)\bigr)U(t,x)\de_x\phi
				-m(c)\de_xU(t,x)\phi + f(t,x,\rho)\phi
				\bigr]
			\Bigr\} \dx\dt \\
			&\geq -\int_0^T\!\!\int_\setR \abs{\phi}\d\mu_{0,t}(x)\dt
			-\int_0^T\!\!\int_\setR \abs{\de_x\phi}\d\mu_{1,t}(x)\dt
		\end{split}
	\end{equation*}
	holds for every constant $c\geq0$ and non-negative test function
	$\phi\in C^\infty_c\bigl([0,T)\times\setR;[0,\infty)\bigr)$.

	We say that $\rho$ is a $(\nu_0,\nu_1)$--quasi-weak solution of \eqref{eq:pde} if
	\begin{equation*}
		\begin{split}
			\abs*{\int_0^T\!\!\int_\setR \bigl(\rho\de_t\phi+m(\rho)U(t,x)\de_x\phi+f(t,x,\rho)\phi\bigr)\dx\dt}
			\leq
			\int_0^T\!\!\int_\setR \abs{\phi}\d\nu_{0,t}(x)\dt
			+\int_0^T\!\!\int_\setR \abs{\de_x\phi}\d\nu_{1,t}(x)\dt
		\end{split}
	\end{equation*}
	holds for every (signed) test function $\phi\in C^\infty_c\bigl([0,T)\times\setR\bigr)$.
\end{definition}

\begin{proposition}[Quasi-entropy solution]\label{prop:quasi-entropy-solution}
	Given $N\in\setN$, let $\bar\rho^N(t,x)$ the piecewise constant density associated to particles
	$(x_0,\dots,x_N)$ and masses $(q_1,\dots,q_N)$ solving the system of ODEs \eqref{eq:ode},
	starting from an initial condition chosen according to \autoref{lem:initial-datum}.

	Then $\bar\rho^N$ is a $(\mu_0^N,\mu_1^N)$--quasi-entropy and $(\nu_0^N,\nu_1^N)$--quasi-weak solution to
	\eqref{eq:pde} according to \autoref{def:quasi-entropy-solution}, with measures
	$\mu_{0,t}^N=\nu_{0,t}^N=0$ and $\mu_{1,t}^N$ and $\nu_{1,t}^N$ enjoying the estimates
	\[
		\mu_{1,t}^N(\setR) \leq \left(\max_{i\in\{1,\dots,N\}}q_i^N(0)\right)H(t),
		\qquad
		\nu_{1,t}^N(\setR) \leq 2\left(\max_{i\in\{1,\dots,N\}}q_i^N(0)\right)H(t),
	\]
	for some increasing function $H:[0,\infty)\to[0,\infty)$ independent of $N$.
\end{proposition}

\begin{proof}
	Let $\bar\rho=\bar\rho^N$ to simplify notation.
	The left hand side of the quasi-entropy inequality is
	\begin{equation}\label{eq:quasi-entropy-lhs}
		\begin{split}
			&\int_0^T\!\!\int_\setR \Bigl\{
			\abs{\bar\rho-c}\de_t\phi + \sign(\bar\rho-c)\bigl[\bigl(m(\bar\rho)-m(c)\bigr)U(t,x)\de_x\phi
				-m(c)\de_xU(t,x)\phi + f(t,x,\bar\rho)\phi
				\bigr]
			\Bigr\} \dx\dt.
		\end{split}
	\end{equation}
	The first addend can be rewritten as
	\[
		\begin{split}
			&\int_0^T\int_\setR \abs{\bar\rho-c}\de_t\phi \dx\dt \\
			&= \int_0^T \left\{
			\sum_{i=1}^{N} \abs{\rho_i-c} \int_{x_{i-1}}^{x_i} \de_t\phi \dx
			+ c \int_{-\infty}^{x_0} \de_t\phi \dx
			+ c \int_{x_N}^\infty \de_t\phi \dx
			\right\} \dt \\
			&= -\int_0^T \sum_{i=1}^N \sign(\rho_i-c)\rho_i' \int_{x_{i-1}}^{x_i} \phi \dx\dt
			-\int_0^T \sum_{i=1}^N \abs{\rho_i-c}\bigl(\phi(x_i)x_i'-\phi(x_{i-1})x_{i-1}'\bigr) \dt
			\spliteq +c\int_0^T \bigl(\phi(x_N)x_N'-\phi(x_0)x_0'\bigr) \dt \\
			&\overset{\eqref{eq:ode-rho}}= \int_0^T \sum_{i=1}^N \sign(\rho_i-c) \rho_i \left[
			(x_i'-x_{i-1}') \dashint_{x_{i-1}}^{x_i} \phi \dx
			- \bigl(\phi(x_i)x_i'-\phi(x_{i-1})x_{i-1}'\bigr) \right] \dt
			\spliteq +c \int_0^T \sum_{i=1}^N \sign(\rho_i-c) \bigl(\phi(x_i)x_i'-\phi(x_{i-1})x_{i-1}'\bigr)
			+c\int_0^T \bigl(\phi(x_N)x_N'-\phi(x_0)x_0'\bigr) \dt
			\spliteq -\int_0^T \sum_{i=1}^N \sign(\rho_i-c) q_i' \dashint_{x_{i-1}}^{x_i} \phi \dx.
		\end{split}
	\]
	Integrating in space $U(t,x)\de_x\phi+\de_xU(t,x)\phi$, the remaining terms can be rewritten as
	\[
		\begin{split}
			&\int_0^T\!\!\int_\setR \Bigl\{
			\sign(\bar\rho-c)\bigl[\bigl(m(\bar\rho)-m(c)\bigr)U(t,x)\de_x\phi
				-m(c)\de_xU(t,x)\phi + f(t,x,\bar\rho)\phi
				\bigr]
			\Bigr\} \dx\dt \\
			&= \sum_{i=1}^N \int_0^T \sign(\rho_i-c) \int_{x_{i-1}}^{x_i} \bigl[\rho_i v(\rho_i) U(t,x) \partial_x \phi + f(t,x,\rho_i) \phi\bigr] \dx \dt
			\spliteq - cv(c) \sum_{i=1}^N \int_0^T \sign(\rho_i-c)
			\bigl[ U(t,x_i)\phi(x_i) - U(t,x_{i-1})\phi(x_{i-1}) \bigr]  \dt
			\spliteq - cv(c) \int_0^T
			\bigl[ \sign(\rho_0 -c)U(t,x_0)\phi(x_0) - \sign(\rho_N -c)U(t,x_N)\phi(x_N) \bigr] \dt .
		\end{split}
	\]
	Combining these two computations we can write \eqref{eq:quasi-entropy-lhs} as the sum of two terms $I + \II$ where
	\[
		\begin{split}
			I &=  c \int_0^T \sum_{i=1}^N \sign(\rho_i-c) \bigl(\phi(x_i)x_i' - \phi(x_{i-1})x_{i-1}'\bigr)
			+c\int_0^T \bigl(\phi(x_N)x_N'-\phi(x_0)x_0'\bigr) \dt
			\spliteq - cv(c) \sum_{i=1}^N \int_0^T \sign(\rho_i -c) \bigl[U(t,x_i)\phi(x_i) - U(t,x_{i-1})\phi(x_{i-1}) \bigr] \dt
			\spliteq - cv(c) \int_0^T \bigl[ \sign(\rho_0 -c)U(t,x_0)\phi(x_0) - \sign(\rho_N -c)U(t,x_N)\phi(x_N) \bigr] \dt  , \\
			\II &= \int_0^T \sum_{i=1}^N \sign(\rho_i-c) \rho_i \left[
			(x_i'-x_{i-1}') \dashint_{x_{i-1}}^{x_i} \phi \dx
			- \bigl(\phi(x_i)x_i'-\phi(x_{i-1})x_{i-1}'\bigr) \right] \dt
			\spliteq -\int_0^T \sum_{i=1}^N \sign(\rho_i-c) q_i' \dashint_{x_{i-1}}^{x_i} \phi \dx
			\spliteq +\sum_{i=1}^N \int_0^T \sign(\rho_i-c) \int_{x_{i-1}}^{x_i} \bigl[\rho_i v(\rho_i) U(t,x) \partial_x \phi + f(t,x,\rho_i) \phi\bigr] \dx \dt.
		\end{split}
	\]

	We claim that $I \geq 0$ and that $\II$ vanishes as $N \to \infty$.
	Using \eqref{eq:ode-x} we can compute
	\[
		\begin{split}
			I &=  c \int_0^T \sum_{i=1}^N \sign(\rho_i-c)
			\bigl[ \phi(x_i)\bigl(v_i-v(c)\bigr)U(t,x_i) - \phi(x_{i-1}) \bigl(v_{i-1}-v(c)\bigr)U(t,x_{i-1}) \bigr]
			\spliteq +c\int_0^T \bigl[ \sign(\rho_0-c)\bigl(v_0-v(c)\bigr)U(t,x_0)\phi(x_0)
				-\sign(\rho_N-c)\bigl(v_N-v(c)\bigr)U(t,x_N)\phi(x_N) \bigr]\dt \\
			&= -c \int_0^T \sum_{i=0}^N [\sign(\rho_{i+1}-c)-\sign(\rho_i-c)]
			\bigl(v_i-v(c)\bigr) U(t,x_i) \phi(x_i) \dt,
		\end{split}
	\]
	which is non-negative thanks to \eqref{eq:good-v-c} of \autoref{lem:good-v}.

	The rest of the proof is devoted to showing that $\II\to0$.
	We split $\II=\II_{\mathrm{adv}}+\II_{\mathrm{src}}$ where
	\begin{align*}
		\II_{\mathrm{adv}} & =
		\int_0^T \sum_{i=1}^N \sign(\rho_i-c) \rho_i \left[ (x_i'-x_{i-1}') \dashint_{x_{i-1}}^{x_i} \phi \dx
		- \bigl(\phi(x_i)x_i'-\phi(x_{i-1})x_{i-1}'\bigr) \right] \dt
		\spliteq +\int_0^T \sum_{i=1}^N \sign(\rho_i-c) \int_{x_{i-1}}^{x_i} \rho_i v(\rho_i) U(t,x) \partial_x \phi  \dx \dt, \\
		\II_{\mathrm{src}} & =
		-\int_0^T \sum_{i=1}^N \sign(\rho_i-c) q_i' \dashint_{x_{i-1}}^{x_i} \phi \dx
		+\int_0^T \sum_{i=1}^N \sign(\rho_i-c) \int_{x_{i-1}}^{x_i}  f(t,x,\rho_i) \phi \dx \dt.
	\end{align*}
	The term $\II_{\mathrm{adv}}$ can be rewritten as
	\[
		\begin{split}
			\II_{\mathrm{adv}}
			&= \int_0^T \sum_{i=1}^N \sign(\rho_i-c) \rho_i
			(x_i'-x_{i-1}') \left(\dashint_{x_{i-1}}^{x_i} \phi \dx - \phi(x_{i-1})\right) \dt
			\spliteq -\int_0^T \sum_{i=1}^N \sign(\rho_i-c) \rho_i x_i' \bigl(\phi(x_i)-\phi(x_{i-1})\bigr) \dt
			\spliteq +\int_0^T \sum_{i=1}^N \sign(\rho_i-c) \int_{x_{i-1}}^{x_i} \rho_i v(\rho_i) U(t,x) \partial_x \phi  \dx \dt, \\
			&= \int_0^T \sum_{i=1}^N \sign(\rho_i-c) \rho_i
			(x_i'-x_{i-1}') \left(\dashint_{x_{i-1}}^{x_i} \phi \dx - \phi(x_{i-1})\right) \dt
			\spliteq +\int_0^T \sum_{i=1}^N \sign(\rho_i-c)\rho_i
			\int_{x_{i-1}}^{x_i} \bigl[v(\rho_i)U(t,x) - v_iU(t,x_i)\bigr]\de_x\phi \dx\dt.
		\end{split}
	\]
	Observe that
	\[
		\begin{split}
			&\abs{x_i'-x_{i-1}'} \\
			&= \abs{v_iU_i-v_{i-1}U_{i-1}}
			\leq \abs{v(\rho_i)(U_i-U_{i-1})} + \abs{(v_i-v(\rho_i))U_i} - \abs{(v_{i-1}-v(\rho_i))U_{i-1}} \\
			&\overset{\eqref{eq:first-diff-U}}\leq v(\rho_i) [C_1+C_2\rho_i](x_i-x_{i-1})
			+\norm{v'}_{L^\infty([0,R(t)])} \bigl(\abs{\rho_i-\rho_{i-1}}\cdot\abs{U_{i-1}}
			+\abs{\rho_{i+1}-\rho_i}\cdot\abs{U_i}\bigr) \\
			&\leq \norm{v}_\infty [C_1+C_2\rho_i](x_i-x_{i-1})
			+\norm{v'}_{L^\infty([0,R(t)])} \left(\max_{i\in\{0,\dots,N\}}\abs{U_i}\right)
			\bigl(\abs{\rho_i-\rho_{i-1}}+\abs{\rho_{i+1}-\rho_i}\bigr)
		\end{split}
	\]
	and\[
		\begin{split}
			&\abs{v(\rho_i)U(t,x)-v_iU(t,x_i)} \\
			&\leq \abs{v(\rho_i)}\abs{U(t,x)-U(t,x_i)} + \abs{v(\rho_i)-v_i}\cdot\abs{U(t,x_i)} \\
			&\leq \norm{v}_\infty \norm{\de_xU}_{L^\infty(\supp\bar\rho)} (x_i-x)
			+ \norm{v'}_{L^\infty([0,R(t)])} \abs{\rho_{i+1}-\rho_i}
			\left(\max_{i\in\{0,\dots,N\}}\abs{U_i}\right).
		\end{split}
	\]
	From $U=V-\de_xW*\bar\rho$ and $\de_xU = \de_xV-D_x\de_xW*\bar\rho = \de_xV - \de_x^2W*\bar\rho-w\bar\rho$,
	using the assumptions \ref{as:V} and \ref{as:W}
	and \autoref{prop:bound-mass}, \autoref{prop:bound-support}, \autoref{prop:bound-density}
	we get
	\[
		\begin{split}
			\max_{i\in\{0,\dots,N\}}\abs{U_i}
			\leq \norm{U}_{L^\infty(\supp\bar\rho)}
			\leq F(t)\bigl[G\bigl(S(t)\bigr) + G\bigl(2S(t)\bigr)q(0)Q(t)\bigr]
		\end{split}
	\]
	and
	\[
		\begin{split}
			\norm{\de_xU}_{L^\infty(\supp\bar\rho)}
			&\leq \norm{\de_xV}_{L^\infty([-S(t),S(t)])}
			+ \norm{\de_x^2W}_{L^\infty([-2S(t),2S(t)])}\norm{\bar\rho}_1 + \abs{w}\norm{\bar\rho}_\infty \\
			&\leq F(t)\bigl[G\bigl(S(t)\bigr) + G\bigl(2S(t)\bigr)q(0)Q(t) + R(t)\bigr],
		\end{split}
	\]
	hence we can estimate (using also \autoref{rmk:diff-U-constants})
	\[
		\begin{split}
			&\abs{x_i'-x_{i-1}'} + \norm{v(\rho_i)U(t,\plchldr)-v_iU(t,x_i)}_{L^\infty([x_{i-1},x_i])} \\
			&\leq \norm{v}_\infty \bigl[C_1(t)+C_2(t)R(t)+\norm{\de_xU}_{L^\infty(\supp\bar\rho)}\bigr]
			\bigl(x_i-x_{i-1}\bigr)
			\spliteq +\norm{v'}_{L^\infty([0,R(t)])} \left(\max_{i\in\{0,\dots,N\}}\abs{U_i}\right)
			\bigl(\abs{\rho_i-\rho_{i-1}}+2\abs{\rho_{i+1}-\rho_i}\bigr) \\
			&\leq 2\norm{v}_\infty F(t)\bigl[G\bigl(S(t)\bigr)+G\bigl(2S(t)\bigr)q(0)Q(t)+R(t)\bigr]
			(x_i-x_{i-1})
			\spliteq +F(t)G\bigl(R(t)\bigr)\bigl[G\bigl(S(t)\bigr)+G\bigl(2S(t)\bigr)q(0)Q(t)\bigr]
			\bigl(\abs{\rho_i-\rho_{i-1}}+2\abs{\rho_{i+1}-\rho_i}\bigr) \\
			&\leq F(t) \bigl[2\norm{v}_\infty+G\bigl(R(t)\bigr)\bigr]
			\bigl[G\bigl(S(t)\bigr)+G\bigl(2S(t)\bigr)q(0)Q(t)+R(t)\bigr]
			\spliteq\qquad \cdot\bigl((x_i-x_{i-1})+\abs{\rho_i-\rho_{i-1}}+2\abs{\rho_{i+1}-\rho_i}\bigr).
		\end{split}
	\]
	This allows us to deduce that
	\[
		\begin{split}
			\abs{\II_{\mathrm{adv}}}
			&\leq \int_0^T \sum_{i=1}^N \rho_i
			\bigl(\abs{x_i'-x_{i-1}'}
			+ \norm{v(\rho_i)U(t,x)-v_iU(t,x_i)}_{L^\infty([x_{i-1},x_i])}\bigr)
			\int_{x_{i-1}}^{x_i} \abs{\de_x\phi} \dx \dt \\
			&= \int_0^T \!\! \int_\setR \abs{\de_x\phi(x)} \d\mu_{1,t}^{\mathrm{adv}}(x) \dt
		\end{split}
	\]
	where the measure
	\[
		\begin{split}
			\mu_{1,t}^{\mathrm{adv}}
			&= \sum_{i=1}^N \rho_i(t) \Bigl(\abs{x_i'(t)-x_{i-1}'(t)}
			+\norm*{v(\rho_i(t))U(t,x)-v_i(t)U\bigl(t,x_i(t)\bigr)}_{L^\infty([x_{i-1}(t),x_i(t)])}\Bigr)
			\bm1_{[x_{i-1}(t),x_i(t)]} \leb^1
		\end{split}
	\]
	has mass controlled by
	\[
		\begin{split}
			&\mu_{1,t}^{\mathrm{adv}}(\setR) \\
			&\leq \sum_{i=1}^N \rho_i(t) \left(\abs{x_i'(t)-x_{i-1}'(t)}
			+ \norm*{v(\rho_i(t))U(t,x)-v_i(t)U\bigl(t,x_i(t)\bigr)}_{L^\infty([x_{i-1}(t),x_i(t)])}\right)
			[x_i(t)-x_{i-1}(t)] \\
			&\leq \sum_{i=1}^N q_i(t) \left(\abs{x_i'(t)-x_{i-1}'(t)}
			+ \norm*{v(\rho_i(t))U(t,x)-v_i(t)U\bigl(t,x_i(t)\bigr)}_{L^\infty([x_{i-1}(t),x_i(t)])}\right) \\
			&\leq \left(\max_{i\in\{1,\dots,N\}} q_i(t)\right)
			F(t) \bigl[2\norm{v}_\infty+G\bigl(R(t)\bigr)\bigr]
			\bigl[G\bigl(S(t)\bigr)+G\bigl(2S(t)\bigr)q(0)Q(t)+R(t)\bigr]
			\spliteq\hspace{2cm} \cdot\sum_{i=1}^N \bigl((x_i-x_{i-1})+\abs{\rho_i-\rho_{i-1}}+2\abs{\rho_{i+1}-\rho_i}\bigr) \\
			&\leq \left(\max_{i\in\{1,\dots,N\}} q_i(0)\right)
			Q(t) F(t) \bigl[2\norm{v}_\infty+G\bigl(R(t)\bigr)\bigr]
			\spliteq\hspace{5cm} \cdot\bigl[G\bigl(S(t)\bigr)+G\bigl(2S(t)\bigr)q(0)Q(t)+R(t)\bigr]
			[2S(t)+3B(t)].
		\end{split}
	\]
	On the other hand
	\[
		\begin{split}
			\abs{\II_{\mathrm{src}}}
			&= \abs*{\int_0^T \sum_{i=1}^N \sign(\rho_i-c) \int_{x_{i-1}}^{x_i}
				f(t,x,\rho_i) \left(\phi(x) -\dashint_{x_{i-1}}^{x_i} \phi(y) \dy \right) \dx\dt} \\
			&\leq \int_0^T \sum_{i=1}^N \int_{x_{i-1}}^{x_i} \abs{f(t,x,\rho_i)}
			\int_{x_{i-1}}^{x_i} \abs{\de_x\phi(y)} \dy \dx\dt
			= \int_0^T \!\! \int_\setR \abs{\de_x\phi(y)} \d\mu_{1,t}^{\mathrm{src}}(y)\dt
		\end{split}
	\]
	where, thanks to assumption \ref{as:f}, the measure
	\[
		\mu_{1,t}^{\mathrm{src}}
		= \sum_{i=1}^N \left(\int_{x_{i-1}(t)}^{x_i(t)} \abs*{f\bigl(t,x,\rho_i(t)\bigr)} \dx\right)
		\bm1_{[x_{i-1}(t),x_i(t)]} \leb^1
	\]
	has mass controlled by
	\[
		\begin{split}
			\mu_{1,t}^{\mathrm{src}}(\setR)
			&\leq \sum_{i=1}^N c_f F(t) \rho_i(t) \bigl(x_i(t)-x_{i-1}(t)\bigr)^2
			\leq \sum_{i=1}^N c_f F(t) q_i(t) \bigl(x_i(t)-x_{i-1}(t)\bigr) \\
			&\leq \left(\max_{i\in\{1,\dots,N\}} q_i(0)\right)
			c_f F(t) Q(t) \sum_{i=1}^N \bigl(x_i(t)-x_{i-1}(t)\bigr)
			\leq 2 \left(\max_{i\in\{1,\dots,N\}} q_i(0)\right) c_f F(t) Q(t) S(t).
		\end{split}
	\]
	Putting together the estimates for $\II_{\mathrm{adv}}$ and $\II_{\mathrm{src}}$ involving
	$\mu_{1,t}^{\mathrm{adv}}$ and $\mu_{1,t}^{\mathrm{srs}}$, we conclude that $\bar\rho^N$ is a
	$(\mu_{0,t}^N,\mu_{1,t}^N)$--quasi-entropy solution with $\mu_{0,t}^N=0$ and
	$\mu_{1,t}^N = \mu_{1,t}^{\mathrm{adv}} + \mu_{1,t}^{\mathrm{srs}}$ with the mass estimated as in the
	claim.

	The fact that $\bar\rho$ is a quasi-weak solution with $\nu_0=\mu_0=0$ and $\nu_1=2\mu_1$
	follows with a standard argument by combining the quasi-entropy inequalities with $c=0$ and $c=R(T)$.
	See the end of the proof of \cite[Proposition~2.13]{Radici-Stra-2023} for the details.
\end{proof}

\begin{proposition}\label{prop:entropy-solution}
	Let $v$, $V$, $W$, $f$ satisfy the \autoref{ass}.
	For every $N\in\setN_+$ let $\bar\rho^N$ be the piecewise constant density
	associated to particles $(x_0,\dots,x_N)$ and masses $(q_1,\dots,q_N)$ solving \eqref{eq:ode},
	with initial condition given by \autoref{lem:initial-datum}.
	Let $\rho\in L^\infty_\loc([0,\infty)\times\setR\bigr)$ be a limit provided by \autoref{prop:compactness}.
	Then $\rho$ is an entropy solution to \eqref{eq:pde} in the sense of \autoref{def:entropy-solution}.
\end{proposition}

Since the proof follows a similar structure as the one of \cite[Proposition~2.14]{Radici-Stra-2023},
we only sketch the main steps.
\begin{proof}
	To avoid confusion, let us denote with $U^N=V-\de_xW*\bar\rho^N$ and $U=V-\de_xW*\rho$
	the free velocity fields associated to $\bar\rho^N$ and $\rho$ respectively.
	We first show that $\rho$ is a weak solution, namely that for every $\phi\in C^\infty_c\bigl((0,T)\times\setR\bigr)$ we have
	\[
		\int_0^T \!\! \int_\setR \bigl[ \rho\de_t\phi + \rho v(\rho)U(t,x)\de_x\phi + f(t,x,\rho)\phi \bigr] \dx\dt
		= 0.
	\]
	It is sufficient to show that the left hand sides of the quasi-weak inequalities for $\bar\rho^N$ pass
	to the limit, because by \autoref{prop:quasi-entropy-solution} the right hand sides are controlled by
	measures $\nu^N$ with vanishing masses.
	From \autoref{prop:compactness} we have that $\bar\rho^N\to\rho$ in $L^1([0,T]\times\setR)$.
	From the Lipschitz regularity of $f$ with respect to $\rho$ granted by the assumption \ref{as:f},
	it is immediate to see that
	\[
		\lim_{N\to\infty}
		\int_0^T \!\! \int_\setR \bigl[\bar\rho^N\de_t\phi + f(t,x,\bar\rho^N)\phi\bigr] \dx\dt
		=
		\int_0^T \!\! \int_\setR \bigl[\rho\de_t\phi + f(t,x,\rho)\phi\bigr] \dx\dt.
	\]
	For the remaining term we instead have
	\[
		\begin{split}
			&\abs*{\int_0^T\!\!\int_\setR \bigl[m(\bar\rho^N)U^N(t,x)-m(\rho)U(t,x)\bigr] \de_x\phi \dx\dt} \\
			&\leq \Lip(\phi) \int_0^T\!\!\int_\setR
			\left\{m(\bar\rho^N)\abs{U^N-U}+\abs{m(\bar\rho^N)-m(\rho)}\cdot\abs{U}\right\} \dx\dt \\
			&\leq \Lip(\phi) \int_0^T \norm{m}_{L^\infty([0,R(t)])}
			\norm{\de_xW*(\bar\rho^N-\rho)}_{L^1([-S(t),S(t)])} \dt
			\spliteq +\Lip(\phi) \int_0^T \norm{m'}_{L^\infty([0,R(t)])} \norm{\bar\rho^N-\rho}_1
			\norm{U}_{L^\infty([-S(t),S(t)])} \dt \\
			&\leq \Lip(\phi) \int_0^T \Bigl\{
			\norm{m}_{L^\infty([0,R(t)])} \norm{\de_xW}_{L^\infty([-2S(t),2S(t)])}
			\spliteq\hspace{2.5cm} +\norm{m'}_{L^\infty([0,R(t)])} \norm{U}_{L^\infty([-S(t),S(t)])}
			\Bigr\} \norm{\bar\rho^N-\rho}_1 \dt,
		\end{split}
	\]
	which goes to zero because $\bar\rho^N\to\rho$ in $L^1\bigl([0,T]\times\setR\bigr)$ and all the other norms
	are bounded by
	\begin{align*}
		\norm{m}_{L^\infty([0,R(t)])}            & \leq \norm{v}_\infty R(t),                    \\
		\norm{m'}_{L^\infty([0,R(t)])}           & \leq \norm{v}_\infty + R(t)G\bigl(R(t)\bigr), \\
		\norm{\de_xW}_{L^\infty([-2S(t),2S(t)])} & \leq F(t)G\bigl(2S(t)\bigr),                  \\
		\norm{U}_{L^\infty([-S(t),S(t)])}        & \leq
		F(t)\bigl[G\bigl(S(t)\bigr) + G\bigl(2S(t)\bigr)q(0)Q(t)\bigr].
	\end{align*}
	This proves that $\rho$ is a weak solution.

	Let us now show that also the quasi-entropy inequalities for $\bar\rho^N$ pass to the limit and
	therefore $\rho$ is an entropy solution.
	Clearly $\int_0^T\!\!\int_\setR \abs{\bar\rho^N-c}\de_t\phi\dx\dt \to \int_0^T\!\!\int_\setR \abs{\rho-c}\de_t\phi\dx\dt$ because $\bar\rho^N\to\rho$ in $L^1_\loc$. Then we also have
	\[
		\begin{split}
			&\abs*{\int_0^T\!\!\int_\setR \left[
					\sign(\bar\rho^N-c)\bigl(m(\bar\rho^N)-m(c)\bigr)U^N(t,x)
					- \sign(\rho-c)\bigl(m(\rho)-m(c)\bigr)U(t,x)
					\right]} \de_x\phi \dx\dt \\
			&\leq \Lip(\phi) \int_0^T\!\!\int_\setR
			\abs*{\sign(\bar\rho^N-c)\bigl(m(\bar\rho^N)-m(c)\bigr)-\sign(\rho-c)\bigl(m(\rho)-m(c)\bigr)}
			\cdot\abs{U^N} \dx\dt
			\spliteq +\Lip(\phi) \int_0^T\!\!\int_\setR
			\sign(\rho-c)\bigl(m(\rho)-m(c)\bigr) \abs{U^N-U} \dx\dt,
		\end{split}
	\]
	which goes to zero with the same argument we used before because the function
	$\rho\mapsto\sign(\rho-c)\bigl(m(\rho)-m(c)\bigr)$ is Lipschitz with constant bounded by
	$\norm{m'}_{L^\infty([0,R(t)])}$ and the norms $\abs{U^N}$ and $\abs{U^N-U}$ are estimated as before.
	It is only left to show that $E_N\to E$ where
	\begin{align*}
		E_N & =
		\int_0^T\!\!\int_\setR \sign(\bar\rho^N-c) \bigl[f(t,x,\bar\rho^N)-m(c)\de_xU^N(t,x)\bigr] \phi \dx\dt, \\
		E   & =
		\int_0^T\!\!\int_\setR \sign(\rho-c) \bigl[f(t,x,\rho)-m(c)\de_xU(t,x)\bigr] \phi \dx\dt.
	\end{align*}
	The difference between the two can be written as $E_N-E=I_N+\II_N$, where
	\begin{align*}
		I_N   & = \int_0^T\!\!\int_\setR \sign(\bar\rho^N-c) \left\{
		\bigl[f(t,x,\bar\rho^N)-f(t,x,\rho)\bigr] - m(c)[\de_xU^N(t,x)-\de_xU(t,x)] \right\} \phi \dx\dt, \\
		\II_N & =\int_0^T\!\!\int_\setR
		\bigl(\sign(\bar\rho^N-c)-\sign(\rho-c)\bigr) \cdot[f(t,x,\rho)-m(c)\de_xU(t,x)] \phi \dx\dt.
	\end{align*}
	The term $I_N\to0$ with a slight variation of the argument used before.
	Indeed, the term with $[f(t,x,\bar\rho^N)-f(t,x,\rho)]$ is treated exactly as before exploiting
	the Lipschitzianity of $f$ in its third argument.
	For the other, let $a>0$ be such that $\supp\phi\subset(0,T)\times(-a,a)$. Then
	\[
		\begin{split}
			&\abs*{\int_0^T\!\!\int_\setR \sign(\bar\rho^N-c)m(c)[\de_xU^N-\de_xU] \phi \dx\dt} \\
			&\leq \norm{\phi}_\infty m(c) \int_0^T
			\norm*{\de_x\bigl(\de_xW*(\bar\rho^N-\rho)\bigr)}_{L^1([-a,a])} \dt \\
			&\leq \norm{\phi}_\infty m(c) \int_0^T \left(\norm{\de_x^2W}_{L^1([-a-S(t),a+S(t)])}+\abs{w(t)}\right)
			\norm{\bar\rho^N-\rho}_1 \dt,
		\end{split}
	\]
	which goes to zero too.
	Showing that $\II_N\to0$ is much trickier because $\bigl(\sign(\bar\rho^N-c)-\sign(\rho-c)\bigr)\to0$
	pointwise only in $Z^c$, where
	\[
		Z=\set{(t,x)\in[0,\infty)\times\setR}{\rho(t,x)=c}.
	\]
	In $Z^c$ we can then use Lebesgue dominated convergence, with the domination
	\[
		2\bigl(\abs{f(t,x,\rho)} + m(c)\abs{\de_xU(t,x)}\bigr)\phi \in L^1([0,T]\times\setR).
	\]
	Inside $Z$, as done in \cite[Proposition~2.14]{Radici-Stra-2023} using the Besicovitch covering theorem
	and fine properties of $BV$ functions, we must instead deduce that $f(t,x,\rho)-m(c)\de_xU(t,x)=0$
	almost everywhere. This information is recovered from the fact that $\rho$ is a weak solution (this is
	the reason we needed to prove that first), heuristically because $\de_t\rho=\de_x\rho=0$ a.e.\ in $Z$
	therefore from the continuity equation we formally get
	\[
		\begin{split}
			0
			&= \de_t\rho + \de_x[m(\rho)U(t,x)] - f(t,x,\rho) \\
			&= \de_t\rho + m'(\rho)\de_x\rho U(t,x) + m(\rho)\de_xU(t,x) - f(t,x,\rho) \\
			&= m(c)\de_xU(t,x) - f(t,x,\rho) \qquad\text{a.e.\ in $Z$.} \qedhere
		\end{split}
	\]
\end{proof}

\begin{proof}[Proof of \autoref{thm:main}]
	By \autoref{prop:bound-mass}, \autoref{prop:bound-support}, \autoref{prop:bound-density} and
	\autoref{prop:bound-TV}, the approximations $\bar\rho^N$ enjoy the uniform a priori estimates
	\begin{align*}
		\norm{\bar\rho^N(t,\plchldr)}_{L^1(\setR)} & \leq \norm{\bar\rho^N(0)}_1 Q(t), &
		\bar\rho^N(t,\plchldr)                     & \leq R(t),                          \\
		\supp \bar\rho^N(t,\plchldr)               & \subseteq [-S(t),S(t)],           &
		\TV\bigl(\bar\rho^N(t,\plchldr)\bigr)      & \leq B(t),
	\end{align*}
	for every $t\in[0,\infty)$ for some increasing functions $Q,R,S,B:[0,\infty)\to[0,\infty)$.

	By \autoref{prop:compactness} there exists a limit density $\rho$ such that the approximations
	$\bar\rho^N\to\rho$ in $L^1_\loc\bigl([0,\infty)\times\setR\bigr)$ and $\rho$ inherits the same
	estimates
	\begin{align*}
		\norm{\rho(t,\plchldr)}_{L^1(\setR)} & \leq \norm{\rho_0}_1 Q(t), &
		\rho(t,\plchldr)                     & \leq R(t),                   \\
		\supp \rho(t,\plchldr)               & \subseteq [-S(t),S(t)],    &
		\TV\bigl(\rho(t,\plchldr)\bigr)      & \leq B(t),
	\end{align*}
	for every $t\in[0,\infty)$.
	The regularity in time of $\rho$ is provided by \autoref{cor:continuity}.

	By \autoref{prop:quasi-entropy-solution} the approximations $\bar\rho^N$ are quasi-entropy solutions of
	\eqref{eq:pde} according to \autoref{def:quasi-entropy-solution}, and thanks to
	\autoref{prop:entropy-solution} this information passes to the limit and implies that $\rho$ is an
	actual entropy solution of \eqref{eq:pde} according to \autoref{def:entropy-solution}.
\end{proof}

\phantomsection
\addcontentsline{toc}{section}{\refname}
\printbibliography

\end{document}